\documentclass[12pt]{article}
 \usepackage[margin=1in]{geometry} 
\usepackage{amsmath,amsthm,amssymb,amsfonts}
 \usepackage{amssymb}

\numberwithin{equation}{section}

\newtheorem{definition}{Definition}[section]
\usepackage{enumitem}
\usepackage[utf8]{inputenc}
\usepackage[pagewise]{lineno}

\usepackage{hyperref}
 \usepackage{csquotes}
\usepackage[utf8]{inputenc}
\usepackage[english]{babel}
\usepackage{xcolor}

\providecommand{\keywords}[1]
{
  \small	
  \textbf{\textit{Keywords:}} #1
}

\newcommand{\MSC}[1]{%
  \small
  \textbf{\textit{Mathematics Subject Classification:}} #1
}
\title{  $L^\infty$ bounds and asymptotic behavior in a doubly degenerate chemotaxis system below six dimensions}
\author{
    Minh Le\thanks{ Institute for Theoretical Sciences, Westlake University, China \texttt{(leminh@westlake.edu.cn)}} 
}
\date{}

\begin{document}
\maketitle

\begin{abstract}
 We investigate a doubly degenerate nutrient-taxis system of the form
\begin{equation*}
\begin{cases}
    u_t = \nabla \cdot (u v \nabla u) - \chi \nabla \cdot (u^\alpha v \nabla v) + \ell  u v, \qquad &x \in \Omega, \ t > 0, \\
    v_t = \Delta v - u v, \qquad &x \in \Omega, \ t > 0,
\end{cases}
\end{equation*}
subject to the homogeneous Neumann boundary conditions in a smoothly bounded convex domain $\Omega \subset \mathbb{R}^n$ with $n\in \left \{ 3,4,5 \right \}$,  where $\alpha \geq 1$, $\chi>0$ and $\ell  \geq 0$. For any suitably regular initial data, we establish the global existence of a weak solution that remains uniformly bounded in time, provided that $\alpha$ lies in the range $\left[1, \frac{5}{2} - \frac{n}{4}\right)$, and we also determine the large-time behavior of these solutions. Our proof relies on several novel functional inequalities, a bootstrap argument, and a Moser iteration method.

\end{abstract}

\keywords{Degenerate diffusion; Chemotaxis; Global existence}\\
\MSC{35B35, 35K45, 35K55, 92C15, 92C17}

\numberwithin{equation}{section}
\newtheorem{theorem}{Theorem}[section]
\newtheorem{lemma}[theorem]{Lemma}
\newtheorem{remark}{Remark}[section]
\newtheorem{Prop}{Proposition}[section]
\newtheorem{Def}{Definition}[section]
\newtheorem{Corollary}{Corollary}[theorem]
\allowdisplaybreaks

\section{Introduction}

This paper is concerned with the following system of parabolic equations
\begin{equation} \label{1}
\begin{cases}
    u_t = \nabla \cdot (u v \nabla u) - \chi \nabla \cdot (u^\alpha v \nabla v) + \ell u v, \qquad & x \in \Omega, \ t > 0, \\[2mm]
    v_t = \Delta v - u v, \qquad & x \in \Omega, \ t > 0, \\[2mm]
    (u v \nabla u - \chi u^\alpha v \nabla v) \cdot \nu = \nabla v \cdot \nu = 0, \qquad & x \in \partial \Omega, \ t > 0, \\[2mm]
    u(x,0) = u_0(x), \quad v(x,0) = v_0(x), \qquad & x \in \Omega,
\end{cases}
\end{equation}
posed in a smoothly bounded convex domain $\Omega \subset \mathbb{R}^n$ with $n \in \{3,4,5\}$, where $\alpha \ge 1$, $\chi > 0$, and $\ell \ge 0$.

In the absence of chemotactic effects, i.e., when $\chi = 0$, system \eqref{1} reduces to the original model proposed in~\cite{JCR2013}, which was designed to describe the collective motion of Bacillus subtilis on nutrient-poor agar surfaces, as experimentally observed in~\cite{Fujikawa, Fujikawa1, MMT1992}. In this context, the unknown $u(x,t)$ denotes the bacterial population density, while $v(x,t)$ represents the nutrient concentration at position $x$ and time $t$.  When $\chi > 0$, an additional chemotactic flux is taken into consideration. More precisely, the term $-\chi \nabla \cdot (u^\alpha v \nabla v)$ models the directed movement of bacteria in response to nutrient density. This refined chemotaxis model was also proposed in~\cite{JCR2013}, as it is theoretically argued in \cite{JCR2013, Plaza} to provide a more accurate mathematical description of the underlying biological process. This assertion is further supported by numerical simulations reported in~\cite{JCR2013} for the particular case $\alpha = 2$, which demonstrate qualitative agreement with the experimental observations.

The mathematical understanding of system \eqref{1} remains rather limited; to date, only a sparse collection of results concerning global existence, boundedness, and asymptotic behavior has been established. A striking feature of \eqref{1} lies in its double degeneracy, namely a porous-medium-type degeneracy at $u = 0$ and a cross-degeneracy at $v = 0$, which substantially complicates the analysis by rendering standard parabolic regularity theories inapplicable and necessitating novel techniques to handle the simultaneous loss of ellipticity and the strong coupling between the two equations. Despite these inherent difficulties, several partial results have been obtained in low-dimensional settings. In the one-dimensional case ($n = 1$), it was established in~\cite{WinklerAMS, Li+Winkler} that the system with $\chi > 0$ and $\alpha = 2$ admits a global weak solution that remains uniformly bounded in time, and its asymptotic behavior was also characterized. For the two-dimensional counterpart, the existence of global weak solutions together with their $L^p$-boundedness was first obtained in~\cite{Li2022} under the condition $\alpha \in (1, 3/2)$; this was subsequently improved in~\cite{Winkler-2024}, where it was shown that, under the milder assumption $\alpha < 2$, solutions exist globally and remain bounded in $L^\infty$. In the critical case $\alpha = 2$, global existence and $L^p$-boundedness were demonstrated in~\cite{Winkler2022NA} under a smallness condition on the initial data, while for arbitrary large initial data, it was shown in~\cite{Li+Winkler1} that the presence of quadratic logistic damping of the form $r u - \mu u^2$ suffices to ensure global weak solvability. Remarkably, a very recent work~\cite{ZY2026} establishes the existence of global weak solutions to \eqref{1} and further proves $L^\infty$-bounds for arbitrary large initial data even in the absence of a logistic source, thereby removing the smallness restriction entirely. In a different direction, when the diffusion term in \eqref{1} is replaced by $\nabla \cdot (u^{m-1} v \nabla u)$ with $m \ge 1$, global existence and $L^p$-boundedness were shown in~\cite{Wu2024}, while $L^\infty$-bounds and large-time behavior were subsequently addressed in~\cite{Wu2026}. By contrast, the three-dimensional situation is considerably more challenging, and only a handful of results have been reported so far. Specifically,~\cite{Li2022} establishes global weak solutions and $L^p$-bounds for $\alpha \in (7/6, 13/9)$, whereas the more recent works~\cite{DAY2026, DW2026} obtain an analogous result, additionally proving $L^\infty$-boundedness and describing the long-time dynamics for $\alpha \in (3/2, 19/12)$.

A natural question that arises is whether one can establish global existence of weak solutions and derive uniform $L^\infty$-bounds in a general spatial dimension. One of the first attempts to address this question was undertaken in~\cite{WinklerCVPDE} for the case $\chi = 0$, wherein the author established the existence of weak solutions and further characterized their large-time behavior. Furthermore, if the second equation in \eqref{1} is replaced by
\[
v_t = \Delta v - u^\beta v,
\]
with $\beta < 2/n$ for $n \ge 2$, then the global existence and boundedness of weak solutions, along with their large-time asymptotics for the critical case $\alpha = 2$, have been shown in~\cite{TSD2026}. Moreover, in the presence of logistic damping of the form $r u - \mu u^{\kappa}$ in the first equation of \eqref{1}, it was demonstrated in~\cite{Pan2024} that if $\kappa > \frac{n+2}{2}$, then a global weak solution exists in time. Furthermore, the large-time behavior of such solutions was subsequently established in~\cite{Pan2026}. However, when $\beta = 1$, $\chi>0$ and without logistic source, no study to date has resolved the global existence and boundedness problem in dimensions $n \ge 4$ with $\alpha \geq 1$. The main contribution of our work is to provide a partial answer to this open question. Specifically, we have successfully established the existence of global weak solutions in dimensions $n = 3, 4, 5$ and, moreover, have derived uniform $L^\infty$-bounds for these solutions. Before presenting our main results, we first clarify the assumptions imposed on the initial data, as detailed below:
\begin{equation} \label{ini}
    \begin{cases}
        u_0 \in W^{1, \infty}(\Omega) \, \text{is nonnegative with } u_0 \not\equiv 0 \qquad \text{and }\\
        v_0\in W^{1, \infty}(\Omega) \, \text{is positive in }\Bar{\Omega}.
    \end{cases}
\end{equation}
Let us state our main result in the following theorem.
\begin{theorem} \label{thm1}
    Let $\Omega \subset \mathbb{R}^n$ with $n \in \left \{3,4,5 \right \}$ be a bounded convex domain with smooth boundary, and let $\alpha \in \left [1, \frac{5}{2}-\frac{n}{4} \right )$, $\chi>0$ and $\ell \geq 0$. Then for any choice of initial data satisfying the condition \eqref{ini}, there exist functions 
    \begin{equation} \label{thm1-1}
        \begin{cases}
            u \in C^0\left ( \Bar{\Omega}\times [0, \infty)  \right ), \qquad \text{and }\\
            v \in C^0\left ( \Bar{\Omega}\times [0, \infty)  \right ) \cap C^{2,1}\left ( \Bar{\Omega}\times (0, \infty)  \right ),
        \end{cases}
    \end{equation}
    such that $u \geq 0$ and $v>0$ in $\Bar{\Omega}\times (0, \infty)$, and that $(u,v)$ forms a global weak solution of \eqref{1} in the sense of Definition \ref{Def}. Furthermore, this solution is uniformly bounded in the sense that
    \begin{align} \label{thm1-2}
        \sup_{t>0} \left \{ \left \| u(\cdot,t) \right \|_{L^\infty(\Omega)}+ \left \| v(\cdot,t) \right \|_{W^{1,\infty}(\Omega)} \right \} <\infty.
    \end{align}
\end{theorem}
\begin{remark}
   Our results improve and extend several previous contributions in the literature. In particular, the global existence of weak solutions in dimension three was first established in \cite{Li2022} for the parameter range 
\(
\alpha \in \left( \frac{7}{6}, \frac{13}{9} \right).
\)
Subsequently, the authors of \cite{DAY2026} and \cite{DW2026} obtained global weak solutions and, moreover, proved uniform $L^\infty$ bounds for solutions when
\(
\alpha \in \left( \frac{3}{2}, \frac{19}{12} \right).
\)
In contrast, the present work significantly broadens these ranges. Specifically, we are able to handle all $\alpha \in \left[1, \frac{7}{4} \right)$ with $n=3$.
\end{remark}
\begin{remark}
   To the best of our knowledge, the present work provides the first result concerning the global existence and boundedness of solutions to the system \eqref{1} in dimensions $n \in \left \{4,5 \right \}$.
\end{remark}
\begin{remark}
   We leave open the problem of determining the sharp (or optimal) range of the parameter $\alpha$ in the three-dimensional setting $n=3$.
\end{remark}

Thanks to the uniform $L^\infty$ bound of solutions as established in Theorem \ref{thm1}, we can establish the asymptotic behavior of solutions in the following theorem. 

\begin{theorem} \label{thm2}
    Under the assumptions of Theorem \ref{thm1}, the global weak solution $(u,v)$ of \eqref{1} exhibits the following long-time behavior: there exists $u_\infty \in C^0(\overline{\Omega})$ such that
    \begin{align} \label{thm2-1}
            u(\cdot,t) \to u_\infty \quad \text{and} \quad v(\cdot,t) \to 0 \quad \text{in } L^\infty(\Omega) \text{ as } t\to\infty.
    \end{align}
     Here the limit function satisfies $u_\infty = w(\cdot,1)$ with $w\in C^0 \left ( \Bar{\Omega}\times [0,1] \right )$ being a weak solution of 
    \begin{equation*}
        \begin{cases}
            w_\tau = \nabla \cdot (a(x, \tau ) w \nabla w)- \nabla \cdot (b(x,\tau ) w^\alpha)+\ell a(x,\tau ) w, \quad &x\in \Omega, \tau \in (0,1), \\
            \nabla w \cdot \nu =0\quad &x\in \partial \Omega, \tau \in (0,1), \\
            w(x,0)= w_0(x), \quad&x\in \Omega
        \end{cases}
    \end{equation*}
    in the sense that $w^2 \in L^1_{loc}\left ( [0,1); W^{1,1}(\Omega)\right )$ and that 
    \begin{align} \label{thm2-2}
        - \int_0^1 \int_\Omega w \varphi_\tau - \int_\Omega w_0\varphi(\cdot,0)&= -\frac{1}{2}\int_0^1\int_\Omega a(x, \tau) \nabla w^2 \cdot \nabla \varphi+ \int_0^1\int_\Omega b(x,\tau ) w^\alpha \cdot \nabla \varphi\notag \\&\quad+ \ell \int_0^1\int_\Omega a(x,\tau )w \varphi 
    \end{align}
    for all $\varphi \in C^1_0\left ( \Bar{\Omega}\times [0, 1) \right )$,
   where for $(x,\tau )\in \Omega \times (0,1)$
    \begin{align*}
        a(x,\tau ):= L \cdot \frac{v(x, \phi^{-1}(\tau ))}{\left \| v(\cdot, \tau ) \right \|_{L^\infty(\Omega)} } \quad \text{and }\quad b(x,\tau ):= L \cdot \frac{v(x, \phi^{-1}(\tau )) \nabla v(x, \phi^{-1}(\tau ))}{\left \| v(\cdot, \tau ) \right \|_{L^\infty(\Omega)}} ,
    \end{align*}
    with 
    \begin{align*}
    L:= \int_0^\infty \left \| v(\cdot,s) \right \|_{L^\infty(\Omega)}
    \end{align*}
    and 
    \begin{align*}
        \phi(t):= \frac{1}{L} \cdot \int_0^t \left \| v(\cdot,s) \right \|_{L^\infty(\Omega)}\, ds , \quad t \geq 0 \, \text{and } \tau = \phi(t),
    \end{align*}
    are such that there exists $C>0$ satisfying 
    \begin{align} \label{thm2-3}
        \frac{1}{C} \leq a(x,\tau) \leq C \quad \text{and } \quad |b(x,\tau)| \leq C \quad \text{for all }(x,\tau )\in \Omega \times (0,1).
    \end{align}
\end{theorem}

Our main ideas are motivated by the recent work of \cite{DAY2026}, in which a bootstrap argument was employed to obtain $L^p$ bounds for solutions. However, upon closer inspection, we find that their method can be sharpened significantly. In fact, by refining the analysis in \cite{DAY2026} more carefully, there remains considerable room for improvement regarding the admissible range of the parameter $\alpha$.

To begin with, we follow a standard strategy from \cite{Li2022} and introduce the functional
\begin{align*}
    F(t) = \int_\Omega u(\cdot,t) \ln u(\cdot,t) + a \int_\Omega \frac{|\nabla v(\cdot,t)|^4}{v^3(\cdot,t)} \qquad \text{for all } t>0,
\end{align*}
with some constant $a>0$ to be chosen appropriately. By carefully tracing the time evolution of $F$ along solutions of the regularized system, we are able to show that $F$ remains globally bounded in time. Moreover, as a direct consequence of this boundedness, we obtain the integral estimate
\begin{align} \label{intro-1}
    \int_0^T \int_\Omega u \frac{|\nabla v|^4}{v^3} + \int_0^T \int_\Omega v |\nabla u|^2 \leq C \qquad \text{for all } T>0,
\end{align}
for some constant $C>0$ independent of time, provided that $\alpha$ lies in the range
\[
\alpha \in \left[1, \frac{3}{2} + \frac{1}{n} \right).
\]
This result not only improves upon Lemma 3.3 of \cite{DAY2026}, where a similar estimate was established only in the three-dimensional case with $\alpha \in \left( \frac{3}{2}, \frac{5}{3} \right]$, but also provides a useful estimate in higher dimensions, particularly for $n \in \{4,5\}$. The key reason we are able to extend the admissible range of $\alpha$ is the crucial inequality established in Lemma~\ref{LM}, which plays a central role throughout our analysis.

Next, we apply a bootstrap argument to prove that $u$ is uniformly bounded in $L^{p_0}$ for some $p_0 > \frac{n}{2}$. More precisely, we combine the estimate \eqref{intro-1} with the key inequality from Lemma~\ref{LM} to determine the optimal exponent $p_0$ for which such a bound holds. Our analysis reveals that the largest possible exponent is given by
\begin{align*}
    p_0 = \frac{2n(2-\alpha)}{n-2},
\end{align*}
and it turns out that $p_0 > \frac{n}{2}$ whenever
\[
\alpha \in \left[1, \frac{5}{2} - \frac{n}{4} \right).
\]
In the physically relevant three-dimensional case $n=3$, this yields an $L^{\frac{3}{2}+}$ bound for $u$ for all $\alpha \in \left[1, \frac{7}{4} \right)$. This significantly improves the corresponding result in \cite{DAY2026}, where the authors only obtained a similar bound in the narrower range $\alpha \in \left( \frac{3}{2}, \frac{19}{12} \right)$. Third, we show that the above $L^{p_0}$ bound can be upgraded to $L^p$ bounds for every $p>1$. This is achieved by employing a standard energy estimate, which is now permissible thanks to the key inequality provided by Lemma~\ref{PL2}.  Finally, we derive uniform $L^\infty$ bounds for $u$ by means of a Moser iteration procedure. Such an approach was already developed in \cite{Winkler-2024} for the two-dimensional setting. Here, we extend this method to arbitrary dimensions $n \ge 2$, an extension that is made possible by the refined estimates in Lemma~\ref{Moser}. The $L^\infty$ bounds obtained in this step are essential for the compactness argument and ultimately allow us to pass to the limit in the approximating system, thereby completing the proof of Theorem \ref{thm1}. 

 Finally, in order to establish the asymptotic behavior of solutions, we adopt the methodological framework developed in \cite{Wu2026}, wherein the central strategy consists of exploiting $L^\infty$  estimates to derive a pointwise Harnack inequality. Leveraging this inequality enables us to recast the first equation of system \eqref{1} as a nonlinear degenerate parabolic equation, from which the long-time behavior of the solutions can then be rigorously established.

The rest of this paper is organized as follows. Section~\ref{S2} is dedicated to the preliminary setup: we define global weak solutions of \eqref{1}, introduce the approximating system \eqref{2}, and collect several auxiliary inequalities that will be needed later. In Section~\ref{S3}, we derive the principal a priori estimates for the regularized system \eqref{2}. These estimates are then used in Section~\ref{S4} to prove $L^p$ bounds for the solutions for arbitrary $p>1$. Next, Section~\ref{S5} establishes the $L^\infty$ bounds that allow us to complete the compactness argument and conclude the proof of Theorem \ref{thm1}. Finally, in Section \ref{S6}, we establish the large time behavior of solutions and provide the proof of Theorem \ref{thm2}.

\section{Local existence of a family of regularized systems and several useful inequalities } \label{S2}

Let us begin this section by introducing the meaning of global weak solution of the system \eqref{1}.

\begin{definition} \label{Def}
Let $n \in \{3, 4,5\}$ and $\Omega \subset \mathbb{R}^n$ be a bounded domain with smooth boundary. Suppose that $\chi > 0$, $\ell \geq 0$, $\alpha \geq 1 $, and that $u_0 \in L^1(\Omega)$ and $v_0 \in L^\infty(\Omega)$ satisfy $u_0 \geq 0$ and $v_0 \geq 0$. By a global weak solution $(u, v)$ of \eqref{1} we mean a pair $(u, v)$ of nonnegative functions satisfying
\begin{align} \label{Def.1}
&u \in L^1_{{loc}}(\Omega \times [0, \infty)) \quad \text{and} \quad 
v \in L^\infty_{{loc}}(\Omega \times [0, \infty)) \cap L^1_{{loc}}([0, \infty); W^{1,1}(\Omega))
\end{align}
and
\begin{align} \label{Def.2}
&u^2  \in L^1_{{loc}}([0,\infty); W^{1,1}(\Omega)) \quad \text{and} \quad 
u^\alpha \nabla v \in L^1_{{loc}}(\Omega \times [0, \infty); \mathbb{R}^n),
\end{align}
which are such that
\begin{align} \label{Def.3}
-\int_0^\infty \int_\Omega u \phi_t - \int_\Omega u_0 \phi(\cdot, 0)
&= -\frac{1}{2} \int_0^\infty \int_\Omega  v \nabla u^2 \cdot \nabla \phi + \chi \int_0^\infty \int_\Omega u^\alpha v \nabla v \cdot \nabla \phi + \ell \int_\Omega uv \phi
\end{align}
for all $\phi \in C_0^\infty(\Omega \times [0, \infty))$ fulfilling $\frac{\partial \phi}{\partial \nu} = 0$ on $\partial \Omega \times (0, \infty)$, and that
\begin{align} \label{Def.4}
\int_0^\infty \int_\Omega v \phi_t + \int_\Omega v_0 \phi(\cdot, 0)
= \int_0^\infty \int_\Omega \nabla v \cdot \nabla \phi + \int_0^\infty \int_\Omega u v \phi
\end{align}
for all $\phi \in C_0^\infty(\Omega \times [0, \infty))$.
\end{definition}

With the aim of constructing weak solutions to \eqref{1} in the sense of Definition~\ref{Def}, we adopt an approximation scheme inspired by \cite{Li+Winkler} and \cite{Winkler2022NA}. To this end, we introduce the following regularized system:
\begin{equation} \label{2}
\begin{cases}
    u_{\varepsilon t} = \nabla \cdot (u_\varepsilon v_\varepsilon \nabla u_\varepsilon) - \chi \nabla \cdot (u_\varepsilon^\alpha v_\varepsilon \nabla v_\varepsilon) + \ell u_\varepsilon v_\varepsilon, \qquad & x \in \Omega, \ t>0, \\[4pt]
    v_{\varepsilon t} = \Delta v_\varepsilon - u_\varepsilon v_\varepsilon, \qquad & x \in \Omega, \ t>0, \\[4pt]
    \dfrac{\partial u_\varepsilon}{\partial \nu} = \dfrac{\partial v_\varepsilon}{\partial \nu} = 0, \qquad & x \in \partial\Omega, \ t>0, \\[4pt]
    u_\varepsilon(x,0) = u_0(x) + \varepsilon, \qquad v_\varepsilon(x,0) = v_0(x), \qquad & x \in \Omega,
\end{cases}
\end{equation}
for $\varepsilon \in (0,1)$. \\

The following lemma asserts the local existence of classical solutions to \eqref{2} as well as the corresponding extensibility criteria.

\begin{lemma} \label{local}
    Let $\Omega \subset \mathbb{R}^n$ with $n \in \left \{3,4,5 \right \}$ be a smoothly bounded domain, let $\alpha \geq 1$, and assume \eqref{ini}. Then for each $\varepsilon \in (0,1)$, there exist $T_{\rm max, \varepsilon} \in (0,\infty]$ and at least one pair $(u_\varepsilon, v_\varepsilon)$ of functions 
    \begin{equation} \label{local-1}
        \begin{cases}
            u_\varepsilon \in C^0 \left ( \Bar{\Omega}\times [0,T_{\rm max, \varepsilon}) \right ) \cap C^{2,1} \left ( {\Omega}\times (0,T_{\rm max, \varepsilon}) \right ) \\
              v_\varepsilon \in C^0 \left ( \Bar{\Omega}\times [0,T_{\rm max, \varepsilon}) \right ) \cap C^{2,1} \left ( {\Omega}\times (0,T_{\rm max, \varepsilon}) \right )
        \end{cases}
    \end{equation}
    which are such that $u_\varepsilon>0$ and $v_\varepsilon >0 $ in $\Bar{\Omega}\times (0,T_{\rm max, \varepsilon })$, that $(u_\varepsilon, v_\varepsilon )$ solves \eqref{2} in the classical sense, and that 
    \begin{align} \label{local-2}
        \text{if }T_{\rm max, \varepsilon}< \infty, \quad \text{then }\quad \limsup_{t \to T_{\rm max, \varepsilon}} \left \|u_\varepsilon(\cdot,t) \right \|_{L^\infty(\Omega)} = \infty.
    \end{align}
    Moreover, this solution satisfies 
    \begin{align} \label{local-3}
        \int_\Omega u_{0 } \leq \int_\Omega u_{\varepsilon }(\cdot,t) \leq \int_\Omega u_0 + \ell \int_\Omega v_0+|\Omega| \qquad \text{for all }t \in (0,T_{\rm max, \varepsilon}) \quad \text{and } \varepsilon \in (0,1)
    \end{align}
    and 
    \begin{align} \label{local-4}
        \left \| v_{\varepsilon} (\cdot,t) \right \|_{L^\infty(\Omega)} \leq \left \| v_\varepsilon(\cdot,t_0) \right \|_{L^\infty(\Omega)} \qquad \text{for all } t_0 \geq 0, t \in (t_0,T_{\rm max, \varepsilon}) \quad \text{and } \varepsilon \in (0,1)
    \end{align}
     as well as 
     \begin{align} \label{local-5}
        \int_{t_0}^{T_{\rm max, \varepsilon}} \int_\Omega u_\varepsilon v_\varepsilon  \leq  \int_\Omega v_\varepsilon( \cdot,t_0)  \qquad \text{for all }t_0 \geq 0 \quad \text{and } \varepsilon \in (0,1).
    \end{align}
\end{lemma}

\begin{proof}
The proofs of all statements are a direct adaptation of the argument in Lemma 2.1 of \cite{Winkler2022NA} to the current framework. 
\end{proof}

The following lemma provides an elementary result which will be applied in Lemma \ref{L4} to obtain $L^{p_0}$ bounds for $u_\varepsilon$ with some $p_0> \frac{n}{2}$.
\begin{lemma} \label{D}
    Let $n \in \{3,4,5\}$, $\alpha \in [1, \frac{3}{2} + \frac{1}{n})$, $p_0 = 1$, and define
    \begin{align*}
        p_k = \min \left\{ 1 + \frac{4}{n} p_{k-1}, \; \frac{2}{n} p_{k-1} + 4 - 2\alpha \right\} \qquad \text{for all } k \geq 1.
    \end{align*}
    Then $(p_k)_{k \geq 0}$ is strictly increasing and converges to
    \begin{align*}
        \lim_{k \to \infty} p_k = \frac{2n(2-\alpha)}{n-2}.
    \end{align*}
\end{lemma}

\begin{proof}
    Since $\alpha < \frac{3}{2}+\frac{1}{n}$, we find that 
    \begin{align*}
        p_1 = \min \left \{ 1+\frac{4}{n}, \frac{2}{n}+4-2\alpha \right \} > p_0=1.
    \end{align*}
    By induction, we deduce that $p_{k+1}>p_k$ for any $k \in \mathbb{N}$. Moreover, we also have 
    \begin{align*}
        p_0=1 < \frac{2n(2-\alpha)}{n-2}
    \end{align*}
    because $\alpha < \frac{3}{2}+\frac{1}{n}$. Hence, by using induction, we also have
    \begin{align*}
        p_k < \frac{2n(2-\alpha)}{n-2} \qquad \text{for all } k \in \mathbb{N}.
    \end{align*}
    Therefore, the sequence $(p_k)_{k \geq 0}$ converges to some value $p_*$ satisfying 
    \begin{align*}
        p_* = \min \left \{1+\frac{4}{n}p_*, \frac{2}{n}p_*+4- 2\alpha \right \}.
    \end{align*}
    Solving this equation gives us that 
    \begin{align*}
        p_*=\frac{2n(2-\alpha)}{n-2},
    \end{align*}
    which finishes the proof.
\end{proof}
We turn now to the derivation of a pivotal inequality, which will play an essential role in Lemma~\ref{M} and, more broadly, in the iterative scheme used to prove $L^p$ bounds for $u_\varepsilon$. Its proof is a straightforward consequence of the embedding $W^{1,1}(\Omega) \hookrightarrow L^{\frac{n}{n-1}}(\Omega)$ together with several elementary inequalities.

\begin{lemma} \label{LM}
Let $\Omega \subset \mathbb{R}^n$ with $n \in \left \{3,4,5 \right \}$ be a smoothly bounded domain and let $M>0$ and $1\leq p < \frac{n}{(n-4)_+}$. There exists $C=C(M,p)>0$ such that if $\int_\Omega f^p \leq M$, then 
    \begin{align}
        \int_\Omega f^{1+\frac{4}{n}p}g \leq C \int_\Omega f^{p-1}g |\nabla f|^2 +C \int_\Omega f \frac{|\nabla g|^4}{g^3} +C \int_\Omega f g
    \end{align}
    is valid for all $f\in C^1(\Bar{\Omega})$ and $g \in C^1(\Bar{\Omega}) $ such that $f>0$ and $g>0$ in $\Bar{\Omega}$.
\end{lemma}

\begin{proof}
    Thanks to the continuity of the embedding $W^{1,1}(\Omega) \hookrightarrow  L^{\frac{n}{n-1}}(\Omega)$, we can find $c_1>0$ such that 
    \begin{align*}
       \int_\Omega \varphi^{\frac{n}{n-1}} \leq c_1  \left \{ \int_\Omega |\nabla \varphi| \right \}^\frac{n}{n-1} + c_1 \left \{ \int_\Omega |\varphi| \right \}^\frac{n}{n-1}  \qquad \text{for all } \varphi \in W^{1,1}(\Omega).
    \end{align*}
    Making use of this, we can find $c_2>0$ such that 
    \begin{align} \label{LM.1}
\int_\Omega f^{1+\frac{4}{n}p} g &\leq c_1 \left \{ \int_\Omega \left | \nabla \left ( f^{\frac{(n-1)(n+4p)}{n^2}} g^{\frac{n-1}{n}} \right ) \right |  \right \}^{\frac{n}{n-1}} +c_1 \left \{ \int_\Omega f^{\frac{(n-1)(n+4p)}{n^2}} g^{\frac{n-1}{n}}  \right \}^{\frac{n}{n-1}} \notag \\
         &\leq c_2 \left \{ \int_\Omega f^\frac{4np-n-4p}{n^2} g^{\frac{n-1}{n}} |\nabla f|  \right \}^{\frac{n}{n-1}} + c_2 \left \{ \int_\Omega f^\frac{n^2+4np-n-4p}{n^2}g^{-\frac{1}{n}} |\nabla g|  \right \}^{\frac{n}{n-1}} \notag \\
         &\quad+c_1 \left \{ \int_\Omega f^{\frac{n^2+4np-n-4p}{n^2}} g^{\frac{n-1}{n}}  \right \}^{\frac{n}{n-1}}.
    \end{align}
  Putting $\eta_1 = \frac{1}{4 M^{\frac{4}{3n-4}}}$ and applying H\"older's inequality and Young's inequality, we deduce that 
       \begin{align} \label{LM.2}
            c_2 \left \{ \int_\Omega f^\frac{n^2+4np-n-4p}{n^2}g^{-\frac{1}{n}} |\nabla g|  \right \}^{\frac{n}{n-1}} &\leq c_2 \left \{ \int_\Omega f \frac{|\nabla g|^4}{g^3} \right \}^\frac{n}{4(n-1)} \cdot \left \{ \int_\Omega f^{\frac{3n^2+16np-4n-16p}{3n^2}} g^\frac{3n-4}{3n} \right \}^\frac{3n}{4(n-1)} \notag\\
            &\leq \eta_1 \left \{ \int_\Omega f^{\frac{3n^2+16np-4n-16p}{3n^2}} g^\frac{3n-4}{3n} \right \}^\frac{3n}{3n-4} + c_3 \int_\Omega f \frac{|\nabla g|^4}{g^3},
       \end{align} 
    where $c_3>0$. The condition $1\leq p < \frac{n}{(n-4)_+}$ with $n \in \left \{3,4,5 \right \}$ guarantees that 
    \begin{align*}
        p< \frac{3n^2+16np-4n-16p}{3n^2} < \frac{n+4p}{n}.
    \end{align*}
    One can verify that 
    \begin{align*}
        \frac{3n^2+16np-4n-16p}{3n^2} = \lambda_1 \left (1+\frac{4}{n}p \right ) +p(1-\lambda_1)
    \end{align*}
    with $\lambda_1= \frac{3n-4}{3n}$. Hence, we apply H\"older's inequality to infer that 
    \begin{align*} 
       \eta_1   \left \{ \int_\Omega f^{\frac{3n^2+16np-4n-16p}{3n^2}} g^\frac{3n-4}{3n} \right \}^\frac{3n}{3n-4} &\leq  \eta_1 \left \{ \int_\Omega f^{1+\frac{4}{n}p} \right \}^{\frac{3n\lambda_1}{3n-4}} \cdot\left \{ \int_\Omega f^p \right \}^{\frac{3n(1-\lambda_1)}{3n-4}} \notag \\
        &\leq  \eta_1  M^\frac{4}{3n-4}  \int_\Omega f^{1+\frac{4}{n}p} \notag\\
        &\leq \frac{1}{4} \int_\Omega f^{1+\frac{4}{n}p},
    \end{align*}
    where we have used 
    \begin{equation*}
        \frac{3n(1-\lambda_1)}{3n-4} = \frac{4}{3n-4}.
    \end{equation*}
    This, together with \eqref{LM.2} implies that 
    \begin{align} \label{LM.3}
         c_2 \left \{ \int_\Omega f^\frac{n^2+4np-n-4p}{n^2}g^{-\frac{1}{n}} |\nabla g|  \right \}^{\frac{n}{n-1}} \leq  \frac{1}{4} \int_\Omega f^{1+\frac{4}{n}p}+c_3 \int_\Omega f \frac{|\nabla g|^4}{g^3}.
    \end{align}    
    In light of H\"older's inequality and Young's inequality, we infer that 
    \begin{align} \label{LM.4}
        c_2 \left \{ \int_\Omega f^\frac{4np-n-4p}{n^2} g^{\frac{n-1}{n}} |\nabla f|  \right \}^{\frac{n}{n-1}} &\leq c_2 \left \{ \int_\Omega f^{p-1}g |\nabla f|^2 \right \}^\frac{n}{2(n-1)} \cdot \left \{ \int_\Omega f^{\frac{n^2+8np-n^2p-8p-2n}{n^2}}g^{\frac{n-2}{n}} \right \}^\frac{n}{2(n-1)} \notag \\
        &\leq    \eta_2\left \{ \int_\Omega f^{\frac{n^2+8np-n^2p-8p-2n}{n^2}}g^{\frac{n-2}{n}} \right \}^\frac{n}{n-2}+c_4\int_\Omega f^{p-1}g |\nabla f|^2, 
    \end{align}
    where $\eta_2= \min \left \{ \frac{1}{4}, \frac{1}{4M^2} \right \}$ and $c_4>0$. Setting 
    \[
    \gamma= \frac{n^2+8np-n^2p-8p-2n}{n^2},
    \]
one can verify that $0<\gamma \leq p$ if $n=3$ and $ p \geq \frac{3}{2}$ or $4\leq n \leq  5$. In this case, we invoke H\"older's inequality to deduce that 
\begin{align} \label{LM.5}
\eta_2\left \{ \int_\Omega f^{\frac{n^2+8np-n^2p-8p-2n}{n^2}}g^{\frac{n-2}{n}} \right \}^\frac{n}{n-2} &\leq  \eta_2 \left \{ \int_\Omega f^p  g^{\frac{(n-2)p}{n \gamma}} \right \}^\frac{n \gamma}{p(n-2)} \notag \\
&\leq \eta_2 \left \{ \int_\Omega f^p \right \}^{\frac{\gamma n}{p(n-2)}-1} \cdot \int_\Omega f^p g \notag \\
&\leq  \eta_2 M^{\frac{\gamma n}{p(n-2)}-1} \int_\Omega f^p g. 
\end{align}
    In the remaining case where $n=3$ and $1\leq p < \frac{3}{2}$, we find that $\gamma= \frac{7}{9}p+\frac{1}{3} \in \left (p,1+\frac{4}{3}p \right )$ and therefore have
    \begin{align*}
        \frac{7}{9}p+\frac{1}{3} = \lambda_2 \left  ( 1+\frac{4}{3}p \right )+(1-\lambda_2)p
    \end{align*}
    with $\lambda_2= \frac{3-2p}{3p} \in (0,1)$. Hence, applying H\"older's inequality yields
    \begin{align} \label{LM.6}
         \left \{ \int_\Omega f^{ \frac{7}{9}p+\frac{1}{3}}g^\frac{1}{3}  \right \}^3 \leq \left \{ \int_\Omega f^{1+\frac{4}{3}p} g \right \}^{3\lambda_2} \cdot \left \{ \int_\Omega f^p g^{\frac{1-3\lambda_2}{3-3\lambda_2}} \right \}^{3-3\lambda_2}.
    \end{align}
  In the case $p=1$, we have $3\lambda_2=1$, which, together with \eqref{LM.6}, implies
    \begin{align} \label{LM.7}
       \eta_2 \left \{ \int_\Omega f^{ \frac{7}{9}p+\frac{1}{3}}g^\frac{1}{3}  \right \}^3 &\leq  \eta_2 \left \{ \int_\Omega f^p \right \}^2 \cdot \int_\Omega f^{1+\frac{4}{3}p}g  \notag \\
       &\leq \eta_2 M^2 \int_\Omega f^{1+\frac{4}{3}p}g \notag \\
       &\leq \frac{1}{4}\int_\Omega f^{1+\frac{4}{3}p}g.
    \end{align}
  For $p \in \left(1, \frac{3}{2}\right)$, we obtain $3\lambda_2 < 1$, which, when combined with \eqref{LM.6} and Young's inequality, yields
   \begin{align}\label{LM.8}
          \eta_2 \left \{ \int_\Omega f^{ \frac{7}{9}p+\frac{1}{3}}g^\frac{1}{3}  \right \}^3  &\leq \eta_2 \int_\Omega f^{1+\frac{4}{3}p}g  + \eta_2 \left ( \int_\Omega f^p g^{\frac{1-3\lambda_2}{3-3\lambda_2}} \right )^{\frac{3-3\lambda_2}{1-3\lambda_2}} \notag \\
          &\leq \frac{1}{4} \int_\Omega f^{1+\frac{4}{3}p}g +\eta_2 \left \{ \int_\Omega f^p \right \}^{\frac{2}{1-3\lambda_2}} \cdot \int_\Omega f^p g \notag \\
          &\leq \frac{1}{4} \int_\Omega f^{1+\frac{4}{3}p}g +\eta_2 M^{\frac{2}{1-3\lambda_2}} \int_\Omega f^p g.
   \end{align}
   Combining \eqref{LM.5}, \eqref{LM.7} and \eqref{LM.8}, noting that $1 \leq p<1+\frac{4}{n}p$ when $1\leq p<\frac{n}{(n-4)_+}$ and applying Young's inequality, we obtain 
   \begin{align*}
       \eta_2\left \{ \int_\Omega f^{\frac{n^2+8np-n^2p-8p-2n}{n^2}}g^{\frac{n-2}{n}} \right \}^\frac{n}{n-2} &\leq \frac{1}{4} \int_\Omega f^{1+\frac{4}{n}p}g +c_5 \int_\Omega f^p g \notag \\
       &\leq \frac{3}{8} \int_\Omega f^{1+\frac{4}{n}p}g +c_6 \int_\Omega fg,
       \end{align*}
       where $c_5= \max \left \{ \eta_2 M^{\frac{\gamma n}{p(n-2)}-1}, \eta_2 M^{\frac{2}{1-3\lambda_2}} \right \}$ and $c_6>0$.
   This, together with \eqref{LM.4} implies that 
   \begin{align} \label{LM.9}
        c_2 \left \{ \int_\Omega f^\frac{4np-n-4p}{n^2} g^{\frac{n-1}{n}} |\nabla f|  \right \}^{\frac{n}{n-1}} &\leq  \frac{3}{8} \int_\Omega f^{1+\frac{4}{n}p}g +c_4\int_\Omega f^{p-1}g |\nabla f|^2+ c_6 \int_\Omega f g.
   \end{align}
    Setting $\lambda_3= \frac{4np-n-4p}{4np} \in (0,1)$, we find that 
    \begin{align*}
        \frac{n^2+4np-n-4p}{n^2} = \left (1+\frac{4}{n}p \right )\lambda_3 +(1-\lambda_3)
    \end{align*}
    and $\frac{n \lambda_3}{n-1}<1$. Using this and applying H\"older's inequality and Young's inequality, we obtain that 
   \begin{align}\label{LM.10}
       c_1 \left \{ \int_\Omega f^{\frac{n^2+4np-n-4p}{n^2}} g^{\frac{n-1}{n}}  \right \}^{\frac{n}{n-1}} &\leq \left \{ \int_\Omega f^{1+\frac{4}{n}p} g \right \}^{\frac{n \lambda_3}{n-1}} \cdot \left \{ \int_\Omega fg^{\frac{n-n\lambda_3-1}{n(1-\lambda_3)}} \right \}^{\frac{n(1-\lambda_3)}{n-1}} \notag \\
       &\leq \frac{1}{8}\int_\Omega f^{1+\frac{4}{n}p} g  +c_7 \left \{ \int_\Omega fg^{\frac{n-n\lambda_3-1}{n(1-\lambda_3)}} \right \}^{\frac{n(1-\lambda_3)}{n-n \lambda_3-1}} \notag \\
       &\leq  \frac{1}{8}\int_\Omega f^{1+\frac{4}{n}p} g  +c_7 \left \{ \int_\Omega f \right \}^{\frac{1}{n-n\lambda_3-1}} \cdot  \int_\Omega fg  \notag \\
       &\leq \frac{1}{8}\int_\Omega f^{1+\frac{4}{n}p} g  +c_8 \left \{ \int_\Omega f^p \right \}^{\frac{1}{np-np\lambda_3-p}} \cdot \int_\Omega fg \notag \\
       &\leq  \frac{1}{8}\int_\Omega f^{1+\frac{4}{n}p} g  +c_9 \int_\Omega fg
   \end{align}
   where $c_7,c_8$ and $c_9$ are positive constants. Collecting \eqref{LM.1}, \eqref{LM.3}, \eqref{LM.9} and \eqref{LM.10}, we infer that 
   \begin{align*}
       \frac{1}{4}\int_\Omega f^{1+\frac{4}{n}p} g  \leq c_3 \int_\Omega f \frac{|\nabla g|^4}{g^3}+c_4 \int_\Omega f^{p-1}g |\nabla f|^2 +c_{10}\int_\Omega fg
   \end{align*}
   where $c_{10}=c_6+c_9$, which completes the proof.
    
\end{proof}

Thanks to the Gagliardo-Nirenberg interpolation inequality,  we can derive the following inequality which will later be used in proving Lemma \ref{PL2}.

\begin{lemma}\label{PL1}
    Let $\Omega \subset \mathbb{R}^n$ with $n \geq 3$ be a smoothly bounded domain and suppose that $p \geq  1$ and $1< r< \frac{2n}{n-2}$. Then for any $\eta >0$, there exists $C=C(\eta, p)>0$ such that the following holds
    \begin{align} \label{PL1-1}
        \left \| f^{\frac{p+1}{2}} \sqrt{g} \right \|_{L^r(\Omega)}^2 \leq \eta \int_\Omega f^{p-1} g|\nabla f|^2 + \eta \int_\Omega f^{p+1} \frac{|\nabla g|^2}{g} +C \cdot \left \{ \int_\Omega f \right \}^p  \cdot\left \{ \int_\Omega fg \right \}
    \end{align}
    for all $f\in C^1(\Bar{\Omega})$ and $g\in C^1(\Bar{\Omega})$ such that $f>0$ and $g>0$ in $\Bar{\Omega}$.
\end{lemma}

\begin{proof}
    In light of the Gagliardo-Nirenberg interpolation inequality, we can find $\theta \in (0,1)$ and $c_1>0$ such that 
    \begin{align} \label{PL1.1}
         \left \| f^{\frac{p+1}{2}} \sqrt{g} \right \|_{L^r(\Omega)}^2 \leq c_1 \left \| \nabla \left ( f^{\frac{p+1}{2}} \sqrt{g} \right ) \right \|_{L^2(\Omega)}^{2 \theta} \left \| f^{\frac{p+1}{2}} \sqrt{g} \right \|^{2(1-\theta)}_{L^{\frac{2}{p+1}}(\Omega)} +c_1\left \| f^{\frac{p+1}{2}} \sqrt{g} \right \|^{2}_{L^{\frac{2}{p+1}}(\Omega)}.
    \end{align}
    Setting $\delta= \frac{2\eta }{(p+1)^2}$ and making use of Young's inequality, we deduce that
    \begin{align}\label{PL1.2}
        c_1 \left \| \nabla \left ( f^{\frac{p+1}{2}} \sqrt{g} \right ) \right \|_{L^2(\Omega)}^{2 \theta} \left \| f^{\frac{p+1}{2}} \sqrt{g} \right \|^{2(1-\theta)}_{L^{\frac{2}{p+1}}(\Omega)} &\leq \delta  \left \| \nabla \left ( f^{\frac{p+1}{2}} \sqrt{g} \right ) \right \|_{L^2(\Omega)}^2 +c_2 \left \| f^{\frac{p+1}{2}} \sqrt{g} \right \|^{2}_{L^{\frac{2}{p+1}}(\Omega)} \notag \\
        &\leq \frac{(p+1)^2 \delta }{2} \int_\Omega f^{p-1}g|\nabla f|^2+ \frac{\delta}{2} \int_\Omega f^{p+1} \frac{|\nabla g|^2}{g} \notag\\
        &\quad+c_2 \left \| f^{\frac{p+1}{2}} \sqrt{g} \right \|^{2}_{L^{\frac{2}{p+1}}(\Omega)} \notag \\
        &\leq \eta  \int_\Omega f^{p-1}g|\nabla f|^2+ \eta  \int_\Omega f^{p+1} \frac{|\nabla g|^2}{g} \notag\\
        &\quad+c_2 \left \| f^{\frac{p+1}{2}} \sqrt{g} \right \|^{2}_{L^{\frac{2}{p+1}}(\Omega)} 
    \end{align}
    By H\"older's inequality, we have
    \begin{align}\label{PL1.3}
        \left \| f^{\frac{p+1}{2}} \sqrt{g} \right \|^{2}_{L^{\frac{2}{p+1}}(\Omega)} &= \left \{ \int_\Omega f g^{\frac{1}{p+1}} \right \}^{p+1} \notag \\
        &\leq \left \{ \int_\Omega f \right \}^p \cdot \int_\Omega fg.
    \end{align}
    Collecting \eqref{PL1.1}, \eqref{PL1.2}, \eqref{PL1.3}, we arrive at \eqref{PL1-1}, which completes the proof. 
\end{proof}
To establish the $L^p$ bounds for $u_\varepsilon$ in Lemma~\ref{Lp}, we require the following inequality. Lemma~3.7 in \cite{DAY2026} proves this inequality in the three-dimensional case, and we extend it to arbitrary dimensions using analogous arguments.
\begin{lemma} \label{PL2}
  Let $\Omega \subset \mathbb{R}^n$ with $n \geq 2$ be a smoothly bounded domain.  For $p\geq 1$, $p_0> \frac{n}{2}$,  $L>0$ and $\eta \in (0,1)$, there exists $C=C(\eta, L, p)$ such that whenever $\int_\Omega f^{p_0} \leq L$, 
    \begin{align*}
        \int_\Omega f^{p+2} g \leq \eta \int_\Omega f^{p-1} g |\nabla f|^2 + \eta \int_\Omega f \frac{|\nabla g|^{2p+2}}{g^{2p+1}}+ C\int_\Omega f g
    \end{align*}
    holds for all $f \in C^1(\Bar{\Omega})$ and $g \in C^1(\Bar{\Omega})$ fulfilling $f>0$ and $g>0$ in $\Bar{\Omega}$. 
\end{lemma}

\begin{proof}
    Since $p_0> \frac{n}{2}$, we have that 
    \[
    1<\frac{2p_0}{p_0-1}< \frac{2n}{n-2}.
    \]
    Applying H\"older's inequality and Young's inequality, using the assumption $\int_\Omega f^{p_0} \leq L$, and employing Lemma \ref{PL1}, we obtain 
    \begin{align*} 
        \int_\Omega f^{p+2}g &= \int_\Omega \left ( f^{\frac{p+1}{2}} g^\frac{1}{2} \right )^2 f \notag \\
        &\leq \left \| f^{\frac{p+1}{2}}g^\frac{1}{2} \right \|_{L^{\frac{2p_0}{p_0-1}}}^2 \left \| f \right \|_{L^{p_0}(\Omega)} \notag \\
        &\leq L^{\frac{1}{p_0}}\left \| f^{\frac{p+1}{2}}g^\frac{1}{2} \right \|_{L^{\frac{2p_0}{p_0-1}}}^2 \notag \\
        &\leq \frac{\eta}{2}\int_\Omega f^{p-1} g |\nabla f|^2 + \frac{\eta}{2} \int_\Omega f^{p+1} \frac{|\nabla g|^{2}}{g}+ c_1\int_\Omega f g \notag \\
        &\leq \frac{\eta}{2}\int_\Omega f^{p-1} g |\nabla f|^2 + \frac{\eta}{2} \int_\Omega f \frac{|\nabla g|^{2p+2}}{g^{2p+1}} +\frac{1}{2}\int_\Omega f^{p+2}g + c_1\int_\Omega f g, 
    \end{align*}
    where $c_1>0$. This further leads to 
    \begin{align*}
         \int_\Omega f^{p+2}g \leq \eta \int_\Omega f^{p-1} g |\nabla f|^2 + \eta \int_\Omega f \frac{|\nabla g|^{2p+2}}{g^{2p+1}} +2 c_1\int_\Omega f g, 
    \end{align*}
    which finishes the proof.
\end{proof}

The following lemma extends Lemma 6.2 in \cite{Winkler-2024} to arbitrary dimensions, and its proof is based on the strategy developed in the proof of Lemma 3.1 in \cite{Wu2026}.

\begin{lemma} \label{Moser}
    Let $\Omega \subset \mathbb{R}^n$ with $n \geq 2$ be a smoothly bounded domain and $p_*>n$. Then there exist $\kappa=\kappa(p_*)>0$ and $K=K(p_*)>0$ such that for any choice of $p\geq p_*$ and $\eta \in (0,1]$,
    \begin{align} \label{Moser-1}
        \int_\Omega f^{p+1}g &\leq \eta \int_\Omega f^{p-1}g |\nabla f|^2+ \eta \cdot \left \{ \int_\Omega f^{\frac{p}{2}} \right \}^{\frac{2(p+1)}{p}} \cdot \int_\Omega \frac{|\nabla g|^{2n+2}}{g^{2n+1}} \notag \\
        &\quad+K\eta ^{-\kappa}p^{2 \kappa} \cdot \left \{ \int_\Omega f^{\frac{p}{2}} \right \}^2 \cdot \int_\Omega fg
    \end{align}
    is valid for arbitrary positive functions $f \in C^1(\Bar{\Omega})$ and $g \in C^1(\Bar{\Omega})$.
\end{lemma}
\begin{proof}
    Let \[ q= \frac{2p_*(n+1)}{(n+3)p_*+2},\] we find that 
    \[
    \frac{2n}{n+2}<q<\frac{2n+2}{n+2}
    \]
    because $p_*>n$. In light of the Gagliardo-Nirenberg interpolation inequality, there exists $c_1>0$ such that 
    \begin{align*}
        \left \| \varphi \right \|^2_{L^2(\Omega)} \leq c_1 \left \| \nabla \varphi \right \|^{2\theta}_{L^q(\Omega)}\left \| \varphi \right \|^{2(1-\theta)}_{L^{\frac{2}{3}}(\Omega)} +c_1 \left \| \varphi \right \|^{2}_{L^{\frac{2}{3}}(\Omega)} \qquad \text{for all } \varphi \in C^1(\Bar{\Omega})
    \end{align*}
    where \begin{align*}
        \theta = \frac{2nq}{3nq+2q-2n} \in (0,1).
    \end{align*}
    Making use of this and applying Young's inequality, we obtain  
    \begin{align} \label{Moser.1}
        \int_\Omega f^{p+1}g &= \left \| f^\frac{p+1}{2} g^\frac{1}{2} \right \|_{L^2(\Omega)}^2 \notag \\
        &\leq c_1 \left \| \nabla \left ( f^\frac{p+1}{2} g^\frac{1}{2} \right ) \right \|^{2\theta}_{L^q(\Omega)}\left \|  f^\frac{p+1}{2} g^\frac{1}{2} \right \|^{2(1-\theta)}_{L^{\frac{2}{3}}(\Omega)} +c_1 \left \|  f^\frac{p+1}{2} g^\frac{1}{2} \right \|^{2}_{L^{\frac{2}{3}}(\Omega)} \notag \\
        &\leq \delta \left \| \nabla \left ( f^\frac{p+1}{2} g^\frac{1}{2} \right ) \right \|^{2}_{L^q(\Omega)} + \left ( c_1^{\frac{1}{1-\theta}} \delta^{- \frac{\theta}{1-\theta}} +c_1 \right)\left \|  f^\frac{p+1}{2} g^\frac{1}{2} \right \|^{2}_{L^{\frac{2}{3}}(\Omega)}.
    \end{align}
    By applying several elementary inequalities, we deduce that 
    \begin{align} \label{Moser.2}
         \delta \left \| \nabla \left ( f^\frac{p+1}{2} g^\frac{1}{2} \right ) \right \|^{2}_{L^q(\Omega)} &= \delta \left \| \frac{p+1}{2}f^{\frac{p-1}{2}}g^\frac{1}{2} \nabla f + \frac{1}{2} f^{\frac{p+1}{2}} g^{-\frac{1}{2}} \nabla g \right \|^{2}_{L^q(\Omega)} \notag \\
         &\leq \delta \cdot \left \{ \frac{p+1}{2} \left  \| f^{\frac{p-1}{2}}g^\frac{1}{2} \nabla f  \right \|_{L^q(\Omega)} + \frac{1}{2} \left  \| f^{\frac{p+1}{2}} g^{-\frac{1}{2}} \nabla g  \right \|_{L^q(\Omega)} \right \}^2 \notag \\
         &\leq \frac{(p+1)^2 \delta}{2}\left  \| f^{\frac{p-1}{2}}g^\frac{1}{2} \nabla f  \right \|_{L^q(\Omega)}^2 + \frac{\delta}{2} \left  \| f^{\frac{p+1}{2}} g^{-\frac{1}{2}} \nabla g  \right \|_{L^q(\Omega)}^2.
    \end{align}
    We set 
        \begin{align} \label{Moser.2'}
        c_2 := \max \left \{ 1, |\Omega|^{\frac{2-q}{q}} \right \} \qquad \text{and } \delta := \min \left \{ \frac{\eta |\Omega|^{\frac{q-2}{q}} }{(p+1)^2}, \left ( \frac{\eta }{2 c_2^{n+1}} \right )^\frac{1}{n+1} \right \}.
        \end{align}
    Noting that $q<2$, we employ H\"older's inequality to infer that 
    \begin{align} \label{Moser.3}
         \frac{(p+1)^2 \delta}{2}\left  \| f^{\frac{p-1}{2}}g^\frac{1}{2} \nabla f  \right \|_{L^q(\Omega)}^2 &\leq  \frac{(p+1)^2 \delta}{2} \cdot |\Omega|^{\frac{2-q}{q}}\left  \| f^{\frac{p-1}{2}}g^\frac{1}{2} \nabla f  \right \|_{L^2(\Omega)}^2  \notag \\
         &\leq  \frac{\eta }{2} \int_\Omega f^{p-1}g |\nabla f|^2,
    \end{align}
    according to the first restriction on $\delta$ contained in \eqref{Moser.2'}. We continue to apply H\"older's inequality to obtain that  
    \begin{align} \label{Moser.4}
        \frac{\delta}{2} \left  \| f^{\frac{p+1}{2}} g^{-\frac{1}{2}} \nabla g  \right \|_{L^q(\Omega)}^2 &= \frac{\delta}{2} \cdot \left \{ \int_\Omega f^{\frac{(p+1)q}{2}} g^{-\frac{q}{2}} |\nabla g|^q \right \}^{\frac{2}{q}} \notag \\
        &= \frac{\delta}{2} \cdot \left \{ \int_\Omega \left ( \frac{|\nabla g|^{2n+2}}{g^{2n+1}} \right )^\frac{q}{2n+2} \cdot f^{\frac{(p+1)q}{2}} g^{\frac{nq}{2n+2}} \right \}^\frac{2}{q} \notag \\
        &\leq \frac{\delta}{2} \cdot \left \{ \int_\Omega \frac{|\nabla g|^{2n+2}}{g^{2n+1}}  \right \}^\frac{1}{n+1} \cdot \left \{ \int_\Omega f^{\frac{(n+1)(p+1)q}{2n+2-q}}g^{\frac{nq}{2n+2-q}} \right \}^\frac{2n+2-q}{(n+1)q}.
    \end{align}
    Since $q<\frac{2n+2}{n+2}$, we find that
    \begin{align*}
        \lambda := \frac{qn}{2n+2-q} \in (0,1)
    \end{align*}
   and
    \begin{align*}
        r := \frac{(p+1)q}{2n+2-qn-q} < p+1.
    \end{align*}
    One more application of H\"older's inequality, we have
    \begin{align} \label{Moser.5}
        \left \{ \int_\Omega  f^{\frac{(n+1)(p+1)q}{2n+2-q}}g^{\frac{nq}{2n+2-q}} \right \}^\frac{2n+2-q}{(n+1)q} &= \left \{ \int_\Omega  \left ( f^{p+1}g \right )^\lambda \cdot f^{r(1-\lambda)} \right \}^\frac{2n+2-q}{(n+1)q} \notag \\
        &\leq \left \{ \int_\Omega f^{p+1}g  \right \}^{\frac{(2n+2-q)\lambda}{(n+1)q}} \cdot \left \{ \int_\Omega f^r \right \}^{\frac{(2n+2-q)(1-\lambda)}{(n+1)q}}.
    \end{align}
    One can verify that if $p\geq p_*$ then 
    \begin{align*}
         q= \frac{2p_*(n+1)}{(n+3)p_*+2} \leq \frac{2p(n+1)}{(n+3)p+2},
    \end{align*}
    which further entails that 
    \begin{align*}
        r =\frac{(p+1)q}{2n+2-qn-q} \leq \frac{p}{2}.
    \end{align*}
    Hence, applying H\"older's inequality yields 
    \begin{align} \label{Moser.6}
        \left \{ \int_\Omega f^r \right \}^{\frac{(2n+2-q)(1-\lambda)}{(n+1)q}} &\leq |\Omega |^{\frac{p-2r}{p} \cdot \frac{(2n+2-q)(1-\lambda)}{(n+1)q}} \left \{ \int_\Omega f^\frac{p}{2} \right \}^{\frac{2r(2n+2-q)(1-\lambda)}{pq(n+1)}} \notag \\
        &\leq c_2  \left \{ \int_\Omega f^\frac{p}{2} \right \}^{\frac{2r(2n+2-q)(1-\lambda)}{pq(n+1)}},
    \end{align}
     since 
    \[
    \frac{(2n+2-q)(1-\lambda)}{(n+1)q} = \frac{2-q}{q}.
    \]
    Combining \eqref{Moser.4}, \eqref{Moser.5} and \eqref{Moser.6} and applying Young's inequality and noting that 
    \[
    \frac{(2n+2-q)\lambda}{(n+1)q} = \frac{n}{n+1}
    \]
    and 
    \[
    \frac{2r(2n+2-q)(1-\lambda)}{pq} = \frac{2p+2}{p},
    \]
     we deduce that 
    \begin{align}\label{Moser.7}
          \frac{\delta}{2} \left  \| f^{\frac{p+1}{2}} g^{-\frac{1}{2}} \nabla g  \right \|_{L^q(\Omega)}^2 &\leq \frac{c_2\delta}{2}\left \{ \int_\Omega \frac{|\nabla g|^{2n+2}}{g^{2n+1}}  \right \}^\frac{1}{n+1} \cdot  \left \{ \int_\Omega f^{p+1}g  \right \}^{\frac{(2n+2-q)\lambda}{(n+1)q}} \cdot \left \{ \int_\Omega f^\frac{p}{2} \right \}^{\frac{2r(2n+2-q)(1-\lambda)}{pq(n+1)}} \notag \\
          &\leq   \left \{ \frac{1}{2} \int_\Omega f^{p+1}g  \right \}^{\frac{n}{n+1}} \cdot c_2 \delta \cdot \left \{ \int_\Omega \frac{|\nabla g|^{2n+2}}{g^{2n+1}}  \right \}^\frac{1}{n+1} \cdot \left \{ \int_\Omega f^\frac{p}{2} \right \}^{\frac{2p+2}{p(n+1)}} \notag \\
          &=  \frac{1}{2}\int_\Omega f^{p+1}g + c_2^{n+1} \delta^{n+1}\left \{ \int_\Omega f^\frac{p}{2} \right \}^{\frac{2p+2}{p}} \cdot \int_\Omega \frac{|\nabla g|^{2n+2}}{g^{2n+1}} \notag \\
          &\leq  \frac{1}{2}\int_\Omega f^{p+1}g + \frac{\eta}{2}\left \{ \int_\Omega f^\frac{p}{2} \right \}^{\frac{2p+2}{p}} \cdot \int_\Omega \frac{|\nabla g|^{2n+2}}{g^{2n+1}},
    \end{align}
    where the last inequality holds due to the second restriction on $\delta$ contained in \eqref{Moser.2'}. By H\"older's inequality, we have
    \begin{align}\label{Moser.8}
        \left \|  f^\frac{p+1}{2} g^\frac{1}{2} \right \|^{2}_{L^{\frac{2}{3}}(\Omega)}&= \left \{ \int_\Omega f^{\frac{p+1}{3}}g^\frac{1}{3} \right \}^3 = \left \{ \int_\Omega (fg)^\frac{1}{3} \cdot f^\frac{p}{3} \right \}^3 \leq \left \{ \int_\Omega f^\frac{p}{2} \right \}^2 \cdot \int_\Omega fg.
    \end{align}
    Collecting \eqref{Moser.1}, \eqref{Moser.2}, \eqref{Moser.3}, \eqref{Moser.7} and \eqref{Moser.8}, we arrive at 
    \begin{align}\label{Moser.9}
        \int_\Omega f^{p+1}g &\leq \eta \int_\Omega f^{p-1}g |\nabla f|^2 + \eta\left \{ \int_\Omega f^\frac{p}{2} \right \}^{\frac{2p+2}{p}} \cdot \int_\Omega \frac{|\nabla g|^{2n+2}}{g^{2n+1}}  \notag \\
        &\quad+ 2\left ( c_1^\frac{1}{1-\theta} \delta^{-\frac{\theta}{1-\theta}}+c_1 \right )\left \{ \int_\Omega f^\frac{p}{2} \right \}^2 \cdot \int_\Omega fg.
    \end{align}
    Setting 
    \begin{align*}
        \kappa \equiv \kappa(p_*) := \frac{2nq}{nq+2q-2n} \qquad \text{and } K \equiv K(p_*) := 2 c_1^\frac{3nq+2q-2n}{nq+2q-2n}\max \left \{ \left ( 4|\Omega|^{\frac{2-q}{q}} \right )^\frac{2nq}{nq+2q-2n}, c_2^\frac{2nq}{nq+2q-2n} \right \}
        +2c_1,
    \end{align*}
    we find that 
    \begin{align*}
        \delta^{-1} &= \max \left \{ \frac{(p+1)^2|\Omega |^{\frac{2-q}{1}}}{\eta }, \left ( \frac{2c_2^{n+1}}{\eta } \right )^{\frac{1}{n+1}} \right \} \notag \\
        &\leq \max \left \{ \frac{4p^2|\Omega |^{\frac{2-q}{1}}}{\eta },  \frac{p^2c_2}{\eta }  \right \} \notag \\
        &\leq \eta^{-1}p^2 \max \left \{ 4|\Omega|^{\frac{2-q}{q}}, c_2 \right \},
    \end{align*}
    which entails that 
    \begin{align*}
        2c_1^\frac{1}{1-\theta} \delta^{-\frac{\theta}{1-\theta}} &\leq 2c_1^{\frac{2nq+2q-2n}{nq+2q-2n}} p^{\frac{4nq}{nq+2q-2n}} \eta^{-\frac{2nq}{nq+2q-2n}}\max \left \{ \left ( 4|\Omega|^{\frac{2-q}{q}} \right )^\frac{2nq}{nq+2q-2n}, c_2^\frac{2nq}{nq+2q-2n} \right \}.
    \end{align*}
    Moreover, we have
    \begin{align*}
        2c_1 \leq 2c_1 p^{\frac{4nq}{nq+2q-2n}} \eta^{-\frac{2nq}{nq+2q-2n}}
    \end{align*}
    The two above estimates lead to 
    \begin{align*}
        2\left ( c_1^\frac{1}{1-\theta} \delta^{-\frac{\theta}{1-\theta}}+c_1 \right ) \leq K \eta^{-\kappa}p^{2\kappa}, 
    \end{align*}
    which combines with \eqref{Moser.9} implies \eqref{Moser-1}. The proof is now complete.  
\end{proof}

In order to establish a uniform in time $L^\infty $ bound for $u_\varepsilon$, we will need the following elementary inequality established in \cite{Winkler-2024}[Lemma 6.3].

\begin{lemma} \label{MW}
    Let $a \geq 1, b \geq 1, d \geq 0$ and $(M_k)_{k \in \mathbb{N}} \subset [1, \infty)$ be such that 
    \begin{align}
        M_k \leq a^k M_{k-1}^{2+\frac{d}{2^k}} +b^{2^k} \qquad \text{for all }k \geq 1.
    \end{align}
    Then 
    \begin{align}
        \liminf_{k \to \infty} M_k^{\frac{1}{2^k}} \leq \left ( 2\sqrt{2} a^3b^{1+\frac{d}{2}}M_0 \right)^{e^\frac{d}{2}}.
    \end{align}
\end{lemma}
A useful tool for dealing with the asymptotic stability of positive bounded solutions is provided by the following lemma, the proof of which can be found in Lemma 3.1 of \cite{XM}.
\begin{lemma}\label{Lunif}
    Suppose that $t_0\geq 0$ and $f: (t_0, \infty) \to [0, \infty)$ is a uniformly continuous function such that $\int_{t_0} ^\infty f(t)\, dt <\infty$. Then $f(t) \to 0$ as $t \to \infty$.
\end{lemma}

\section{A Priori Estimates} \label{S3}

In this section, we will establish several useful estimates for solutions of the system \eqref{2} which will be later applied in the sequel sections. Let us begin with the following basic estimate for $v_\varepsilon$.

\begin{lemma}\label{local-5'}
  Let $\Omega \subset \mathbb{R}^n$ with $n \geq 3$ be a bounded domain with smooth boundary and let $\alpha \geq 1$. There exists $C>0$ such that 
  \begin{align*}
     \int_0^{T_{\rm max, \varepsilon}} \int_\Omega v_\varepsilon |\nabla v_\varepsilon|^2 \leq C \qquad \text{for all }\varepsilon \in (0,1).
  \end{align*}
\end{lemma}
\begin{proof}
    Multiplying the second equation of \eqref{2} by $v_\varepsilon^2$ and integrating by parts yields
    \begin{align*}
       \frac{1}{3} \frac{d}{dt}\int_\Omega v_\varepsilon^3 +2 \int_\Omega v_\varepsilon |\nabla v_\varepsilon|^2 &= - \int_\Omega u_\varepsilon v_\varepsilon^3 \notag \\
       &\leq 0 \qquad \text{for all }t\in (0,T_{\rm max, \varepsilon}) \quad \text{and }\varepsilon \in (0,1).
    \end{align*}
    Integrating this over time and noting that $v_\varepsilon>0$, we obtain 
    \begin{align*}
        2 \int_0^{T_{\rm max, \varepsilon}} \int_\Omega v_\varepsilon |\nabla v_\varepsilon|^2 \leq \frac{1}{3} \int_\Omega v_0^3 \qquad \text{for all }\varepsilon \in (0,1),
    \end{align*}
    which completes the proof.
\end{proof}

Let us now derive the following two basic inequalities allowing us to control the time evolution of certain singularly weighted integrals involving the signal gradient.

\begin{lemma} \label{W}
    Let $\Omega \subset \mathbb{R}^n$ with $n \geq 3$ be a bounded convex domain with smooth boundary and $q\geq 2$. Then there exist $C_1(q)>0$, $C_2(q)>0$, $C_3>0$ and $C_4>0$ such that 
    \begin{align}\label{W-1}
        \frac{d}{dt}\int_\Omega \frac{|\nabla v_\varepsilon|^q}{v_\varepsilon^{q-1}} +C_1(q) \int_\Omega \frac{|\nabla v_\varepsilon|^{q+2}}{v_\varepsilon^{q+1}} + C_1(q)\int_\Omega u_\varepsilon \frac{|\nabla v_\varepsilon|^{q}}{v_\varepsilon^{q-1}} \leq C_2(q) \int_\Omega u_\varepsilon^{\frac{q+2}{2}}v_\varepsilon
    \end{align}
    and
     \begin{align} \label{W-2}
        \frac{d}{dt}\int_\Omega \frac{|\nabla v_\varepsilon|^4}{v_\varepsilon^{3}} + C_3\int_\Omega \frac{|\nabla v_\varepsilon|^{6}}{v_\varepsilon^{5}} + C_3\int_\Omega u_\varepsilon \frac{|\nabla v_\varepsilon|^{4}}{v_\varepsilon^{3}} \leq C_4 \int_\Omega  v_\varepsilon|\nabla u_\varepsilon|^{2} 
    \end{align}
    for all $t \in (0,T_{\rm max, \varepsilon})$ and $\varepsilon \in (0,1)$.
\end{lemma}

\begin{proof}
    From Lemma 2.3 in \cite{Wu2024}, there exist $c_1(q)>0$  and $c_2>0$ such that for all $t\in (0,T_{\rm max, \varepsilon})$ and $\varepsilon \in (0,1)$,
    \begin{align} \label{W.1}
        \frac{d}{dt}\int_\Omega \frac{|\nabla v_\varepsilon|^q}{v_\varepsilon^{q-1}} +\frac{q}{2} \int_\Omega \frac{|\nabla v_\varepsilon|^{q-2}}{v_\varepsilon^{q-3}}|D^2 \ln v_\varepsilon|^2 + (q-1)^2\int_\Omega u_\varepsilon \frac{|\nabla v_\varepsilon|^{q}}{v_\varepsilon^{q-1}} \leq c_1(q) \int_\Omega u_\varepsilon^{\frac{q+2}{2}}v_\varepsilon
    \end{align}
    and
     \begin{align} \label{W.2}
        \frac{d}{dt}\int_\Omega \frac{|\nabla v_\varepsilon|^4}{v_\varepsilon^{3}} + \int_\Omega \frac{|\nabla v_\varepsilon|^{2}}{v_\varepsilon}|D^2 \ln v_\varepsilon|^2 + \int_\Omega u_\varepsilon \frac{|\nabla v_\varepsilon|^{4}}{v_\varepsilon^{3}} \leq c_2 \int_\Omega  v_\varepsilon|\nabla u_\varepsilon|^{2}. 
    \end{align}
   Applying Lemma 3.4 in \cite{WinklerDCDSB}, we have
    \begin{align*}
        \int_\Omega \frac{|\nabla v_\varepsilon|^{q+2}}{v_\varepsilon^{q+1}} \leq (q+\sqrt{n})^2  \int_\Omega \frac{|\nabla v_\varepsilon|^{q-2}}{v_\varepsilon^{q-3}}|D^2 \ln v_\varepsilon|^2.
    \end{align*}
   This, together with \eqref{W.1} and \eqref{W.2} implies \eqref{W-1} and \eqref{W-2} with $C_1(q) = \min \left \{ \frac{q}{2(q+\sqrt{n})^2}, (q-1)^2 \right \}$, $C_2(q)=c_1(q)$, $C_3= \frac{1}{(4+\sqrt{n})^2}$ and $C_4=c_2$. The proof is now complete.
\end{proof}

Before proving the most important result of this section---namely, the spatio-temporal estimates for $v_\varepsilon |\nabla u_\varepsilon|^2$ and $u_\varepsilon \frac{|\nabla v_\varepsilon|^4}{v_\varepsilon^3}$---we first establish the following auxiliary estimate.

\begin{lemma} \label{L2}
     Let $\Omega \subset \mathbb{R}^n$ with $n \geq 3$ be a bounded convex domain with smooth boundary and let $\alpha \in [1,2]$. There exists $a>0$ and $b>0$ such that for all $t \in (0,T_{\rm max, \varepsilon})$ and for each $\varepsilon \in (0,1)$, the functional 
    \begin{align*}
        F_\varepsilon(t)= \int_\Omega u_\varepsilon (\cdot,t) \ln u_\varepsilon (\cdot,t) +a \int_\Omega \frac{|\nabla v_\varepsilon(\cdot,t)|^4}{v_\varepsilon(\cdot,t)} 
    \end{align*}
    satisfies 
    \begin{align}\label{L2-1}
        F'_\varepsilon(t)&+b \int_\Omega \frac{|\nabla v_\varepsilon|^6}{v_\varepsilon^5} +b \int_\Omega u_\varepsilon \frac{|\nabla v_\varepsilon|^4}{v_\varepsilon^3} +\frac{1}{2}\int_\Omega v_\varepsilon |\nabla u_\varepsilon|^2 \notag\\
        &\leq \chi^2 \int_\Omega u_\varepsilon^{2\alpha-2} v_\varepsilon |\nabla v_\varepsilon|^2+ \ell \int_\Omega u_\varepsilon v_\varepsilon \ln u_\varepsilon + \ell \int_\Omega u_\varepsilon v_\varepsilon .
    \end{align}
\end{lemma}
\begin{proof}
Testing the first equation of \eqref{2} by $\ln u_\varepsilon+1$ and applying Young's inequality yields
\begin{align} \label{L2.1}
    \frac{d}{dt} \int_\Omega u_\varepsilon \ln u_\varepsilon &=- \int_\Omega v_\varepsilon |\nabla u_\varepsilon|^2+ \chi \int_\Omega u_\varepsilon^{\alpha-1}v_\varepsilon \nabla u_\varepsilon \cdot \nabla v_\varepsilon  +\ell \int_\Omega u_\varepsilon (\ln u_\varepsilon+1) v_\varepsilon \notag \\
    &\leq -\frac{3}{4}\int_\Omega v_\varepsilon |\nabla u_\varepsilon|^2 +\chi^2 \int_\Omega u_\varepsilon^{2\alpha-2} v_\varepsilon |\nabla v_\varepsilon|^2+ \ell \int_\Omega u_\varepsilon v_\varepsilon \ln u_\varepsilon + \ell \int_\Omega u_\varepsilon v_\varepsilon 
\end{align}
for all $t\in (0,T_{\rm max, \varepsilon})$ and $\varepsilon \in (0,1)$. Using \eqref{W-2} in Lemma \ref{W}, we can find $c_1>0$ and $c_2>0$ such that  
 \begin{align} \label{L2.2}
        \frac{d}{dt}\int_\Omega \frac{|\nabla v_\varepsilon|^4}{v_\varepsilon^{3}} + c_1\int_\Omega \frac{|\nabla v_\varepsilon|^{6}}{v_\varepsilon^{5}} + c_1\int_\Omega u_\varepsilon \frac{|\nabla v_\varepsilon|^{4}}{v_\varepsilon^{3}} \leq c_2 \int_\Omega  v_\varepsilon|\nabla u_\varepsilon|^{2} 
    \end{align}
    for all $t\in (0,T_{\rm max, \varepsilon})$ and $\varepsilon \in (0,1)$. Letting $a= \frac{1}{4c_2}$ and $b= ac_1= \frac{c_1}{4c_2}$ and collecting \eqref{L2.1} and \eqref{L2.2}, we arrive at \eqref{L2-1}. The proof is now complete.
\end{proof}

We now turn to the derivation of the key estimates of this section. These bounds will be essential for establishing the $L^{p_0}$-estimate for $u_\varepsilon$, for some $p_0>\frac{n}{2}$, as stated in Lemma~\ref{L4}.
\begin{lemma} \label{L3}
 Let $\Omega \subset \mathbb{R}^n$ with $n \geq 3$ be a bounded convex domain with smooth boundary and assume $\alpha \in \left [1, \frac{3}{2}+\frac{1}{n}\right )$. There exists $C>0$ such that 
    \begin{align} \label{L3-1}
        \int_0 ^{T_{\rm max, \varepsilon}} \int_\Omega u_{\varepsilon} \frac{|\nabla v_{\varepsilon}|^4}{v_{\varepsilon}^3} \leq C \qquad \text{for all } \varepsilon \in (0,1)
    \end{align}
    and
    \begin{align} \label{L3-2}
         \int_0 ^{T_{\rm max, \varepsilon}} \int_\Omega v_{\varepsilon}|\nabla u_{\varepsilon}|^2 \leq C \qquad \text{for all } \varepsilon \in (0,1)
    \end{align}
    as well as
    \begin{align}\label{L3-3}
        \int_0 ^{T_{\rm max, \varepsilon}} \int_\Omega \frac{|\nabla v_\varepsilon|^6}{v_\varepsilon^5} \leq C \qquad \text{for all } \varepsilon \in (0,1).
    \end{align}
\end{lemma}
\begin{proof}
    From Lemma \ref{L2}[\eqref{L2-1}], we have that 
     \begin{align}\label{L3.1}
        F'_\varepsilon(t)&+b \int_\Omega \frac{|\nabla v_\varepsilon|^6}{v_\varepsilon^5} +b \int_\Omega u_\varepsilon \frac{|\nabla v_\varepsilon|^4}{v_\varepsilon^3} +\frac{1}{2}\int_\Omega v_\varepsilon |\nabla u_\varepsilon|^2 \notag\\
        &\leq \chi^2 \int_\Omega u_\varepsilon^{2\alpha-2} v_\varepsilon |\nabla v_\varepsilon|^2+ \ell \int_\Omega u_\varepsilon v_\varepsilon \ln u_\varepsilon + \ell \int_\Omega u_\varepsilon v_\varepsilon ,
    \end{align}
    for all $t\in (0,T_{\rm max, \varepsilon})$ and $\varepsilon \in (0,1)$. In view of Lemma \ref{LM} and \eqref{local-3}, we can find $c_1>0$ such that 
    \begin{align}\label{L3.2}
      \int_\Omega u_\varepsilon^{1+\frac{4}{n}}v_\varepsilon \leq c_{1}\int_\Omega v_\varepsilon |\nabla u_\varepsilon|^2 +c_{1} \int_\Omega u_\varepsilon \frac{|\nabla v_\varepsilon|^4}{v_\varepsilon^3} +c_{1}\int_\Omega u_\varepsilon v_\varepsilon
   \end{align}
   for all $t\in (0,T_{\rm max, \varepsilon})$ and $\varepsilon \in (0,1)$. Putting
   \[
   \eta=  \min \left \{\frac{1}{4c_1(1+\left \|v_0 \right \|^4_{L^\infty(\Omega)})}, \frac{b}{2(c_1+1+c_1\left \|v_0 \right \|^4_{L^\infty(\Omega)}) } \right \},
   \]
   noting that $0 \leq 2\alpha-2 < 1+ \frac{2}{n}$ when $\alpha \in \left [1,\frac{3}{2}+ \frac{1}{n}\right )$, applying Young's inequality and using \eqref{L3.2}, we can find $c_2=c_2(\eta)>0$ and $c_3=c_3(\eta)>0$ satisfying 
   \begin{align}\label{L3.3}
      \chi^2 \int_\Omega u_\varepsilon^{2\alpha-2} v_\varepsilon |\nabla v_\varepsilon|^2+l \int_\Omega u_\varepsilon v_\varepsilon \ln u_\varepsilon  &\leq \eta \int_\Omega u_\varepsilon^{1+\frac{2}{n}} v_\varepsilon |\nabla v_\varepsilon|^2 +c_2 \int_\Omega v_\varepsilon |\nabla v_\varepsilon |^2+\eta \int_\Omega u_\varepsilon^{1+\frac{4}{n}} v_\varepsilon + c_3 \int_\Omega u_\varepsilon v_\varepsilon \notag \\
       &\leq \eta \left (1+ \left \|v_0 \right \|^4_{L^\infty(\Omega)} \right ) \int_\Omega u_\varepsilon^{1+\frac{4}{n}}v_\varepsilon + \eta \int_\Omega u_\varepsilon \frac{|\nabla v_\varepsilon|^4}{v_\varepsilon^3} \notag \\
       &\quad+c_2 \int_\Omega v_\varepsilon |\nabla v_\varepsilon |^2+c_3  \int_\Omega u_\varepsilon v_\varepsilon \notag \\
       &\leq c_1 \eta \left (1+ \left \|v_0 \right \|^4_{L^\infty(\Omega)} \right )\int_\Omega v_\varepsilon |\nabla u_\varepsilon|^2  +c_2\int_\Omega v_\varepsilon |\nabla v_\varepsilon|^2 \notag \\
       &\quad+  \eta \left (c_1+ 1+c_1\left \|v_0 \right \|^4_{L^\infty(\Omega)} \right ) \int_\Omega u_\varepsilon \frac{|\nabla v_\varepsilon|^4}{v_\varepsilon^3} \notag \\
       &\quad+ \left ( \eta c_1(1+ \left \|v_0 \right \|^4_{L^\infty(\Omega)}) +c_3 \right ) \int_\Omega u_\varepsilon v_\varepsilon \notag \\
       &\leq \frac{1}{4} \int_\Omega v_\varepsilon |\nabla u_\varepsilon|^2 +c_2\int_\Omega v_\varepsilon |\nabla v_\varepsilon|^2  \notag \\
       &\quad + \frac{b}{2}\int_\Omega u_\varepsilon \frac{|\nabla v_\varepsilon|^4}{v_\varepsilon^3}+\left (c_3+\frac{1}{4} \right )\int_\Omega u_\varepsilon v_\varepsilon
   \end{align}
     for all $t\in (0,T_{\rm max, \varepsilon})$ and $\varepsilon \in (0,1)$. Combining \eqref{L3.1} and \eqref{L3.3}, we infer that
     \begin{align*}
          F'_\varepsilon(t) +b \int_\Omega \frac{|\nabla v_\varepsilon|^6}{v_\varepsilon^5}+\frac{b}{2} \int_\Omega u_\varepsilon \frac{|\nabla v_\varepsilon|^4}{v_\varepsilon^3} +\frac{1}{4}\int_\Omega v_\varepsilon |\nabla u_\varepsilon|^2  &\leq c_2 \int_\Omega v_\varepsilon |\nabla v_\varepsilon|^2 +c_4 \int_\Omega u_\varepsilon v_\varepsilon 
     \end{align*}
    for all $t\in (0,T_{\rm max, \varepsilon})$ and $\varepsilon \in (0,1)$, where $c_4=c_3+\frac{1}{4}+\ell$. Integrating this over time and making use of \eqref{local-5} and \eqref{local-5'}, we can find $c_5>0$ such that
    \begin{align*}
       b\int_0^t  \int_\Omega \frac{|\nabla v_\varepsilon|^6}{v_\varepsilon^5}+ \frac{b}{2} \int_0^t \int_\Omega u_\varepsilon \frac{|\nabla v_\varepsilon|^4}{v_\varepsilon^3} +\frac{1}{4} \int_0^t \int_\Omega v_\varepsilon |\nabla u_\varepsilon|^2 &\leq c_2 \int_0^t \int_\Omega v_\varepsilon |\nabla v_\varepsilon|^2 +c_4 \int_0^t \int_\Omega u_\varepsilon v_\varepsilon \notag \\
        &\quad + \int_\Omega (u_0+1) \ln (u_0+1) +a \int_\Omega \frac{|\nabla v_0|^4}{v_0^3} + \frac{|\Omega|}{e} \\
        &\leq c_5\qquad \text{for all } t\in (0,T_{\rm max, \varepsilon}) \quad \text{and } \varepsilon \in (0,1),
    \end{align*}
     where the last inequality holds due to $x\ln x \geq -\frac{1}{e}$ for all $x>0$. The proof is now complete.

\end{proof}

\section{\texorpdfstring{$L^p$ boundedness for $u_\varepsilon$ for arbitrary $p>1$}{Lp boundedness for u_epsilon}} \label{S4}
The main goal of this section is to establish $L^p$ bounds for solutions of \eqref{2} for every $p>1$, via an iteration argument. The following lemma is the key ingredient in this iterative procedure: it asserts that if $u_\varepsilon$ is bounded in $L^{p_k}$, then it is also bounded in $L^{p_{k+1}}$ for some sequence $(p_k)_{k \in \mathbb{N}}$ with $p_{k+1}>p_k \geq 1$ and $p_{k} \leq p_*$ for some $p_*>1$.

\begin{lemma} \label{M}
    Let $\Omega \subset \mathbb{R}^n$ with $n \in \left \{ 3,4,5 \right \}$ be a bounded convex domain with smooth boundary and assume that $\alpha \in \left [1, \frac{3}{2}+\frac{1}{n} \right )$ and $1\leq p_0 < \frac{2n(2-\alpha)}{n-2}$. Suppose that there exists $M>0$ such that 
    \begin{align} \label{M-1}
        \int_\Omega u_\varepsilon^{p_0}(\cdot,t) \leq M\qquad \text{for all } t\in (0,T_{\rm max, \varepsilon}) \quad \text{and } \varepsilon \in (0,1)
    \end{align}
    and
    \begin{align} \label{M-2}
        \int_0^{T_{\rm max, \varepsilon}} \int_\Omega u_\varepsilon^{p_0-1} v_\varepsilon |\nabla u_\varepsilon|^2 \leq M \qquad \text{for all } \varepsilon \in (0,1)
        \end{align}
    then for any $1< p \leq \min \left \{ 1+\frac{4}{n}p_0,  \frac{2}{n}p_0+4-2\alpha    \right \}$ there exists $C=C(p)>0$ such that
    \begin{align} \label{M-4}
        \int_\Omega u_\varepsilon^p(\cdot,t) \leq C \qquad \text{for all }t\in (0,T_{\rm max, \varepsilon}) \quad \text{and } \varepsilon \in (0,1)
    \end{align}
    and  
     \begin{align} \label{M-5}
        \int_0^{T_{\rm max, \varepsilon}} \int_\Omega u_\varepsilon^{p-1} v_\varepsilon |\nabla u_\varepsilon|^2 \leq C \qquad \text{for all } \varepsilon \in (0,1).
        \end{align}
\end{lemma}
\begin{proof}
    Multiplying the first equation of \eqref{2} by $u^{p-1}$, integrating by parts, using \eqref{local-4}, noting that $p+2\alpha-3 \leq 1+\frac{2}{n}p_0$ because $p \leq  \frac{2}{n}p_0+4-2\alpha $ and applying Young's inequality yields
    \begin{align} \label{M.1}
        \frac{1}{p} \frac{d}{dt}\int_\Omega u_\varepsilon^p +\frac{p-1}{2}\int_\Omega u_\varepsilon^{p-1}v_\varepsilon |\nabla u_\varepsilon|^2 &\leq \frac{(p-1)\chi^2}{2}\int_\Omega u^{p+2\alpha-3}_\varepsilon v_\varepsilon |\nabla v_\varepsilon|^2 + \ell \int_\Omega u^p_\varepsilon v_\varepsilon   \notag \\
        &\leq \frac{(p-1)\chi^2}{2} \int_\Omega u^{1+\frac{2}{n}p_0}_\varepsilon v_\varepsilon |\nabla v_\varepsilon|^2 +\frac{(p-1)\chi^2}{2} \int_\Omega v_\varepsilon |\nabla v_\varepsilon|^2  \notag \\
        &\quad+ \ell \int_\Omega u_\varepsilon^{1+\frac{4}{n}p_0} v_\varepsilon +\ell \int_\Omega u_\varepsilon v_\varepsilon \notag \\
        &\leq (\ell+1)   \int_\Omega u_\varepsilon^{1+\frac{4}{n}p_0}v_\varepsilon+  c_1\int_\Omega u_\varepsilon \frac{|\nabla v_\varepsilon|^4}{v_\varepsilon^3} +\frac{(p-1)\chi^2}{2} \int_\Omega v_\varepsilon |\nabla v_\varepsilon|^2 \notag \\
        &\quad + \ell \int_\Omega u_\varepsilon v_\varepsilon,
    \end{align}
    for all $t \in (0,T_{\rm max}, \varepsilon)$ and $\varepsilon \in (0,1)$, where $c_1>0$. Applying Lemma \ref{LM} and using \eqref{M-1}, we can find $c_2>0$ such that 
    \begin{align}\label{M.2}
        (\ell+1)   \int_\Omega u_\varepsilon^{1+\frac{4}{n}p_0}v_\varepsilon \leq c_2 \int_\Omega u^{p_0-1}_\varepsilon v_\varepsilon |\nabla u_\varepsilon|^2 + c_2 \int_\Omega u_\varepsilon \frac{|\nabla v_\varepsilon|^4}{v_\varepsilon^3} +c_2 \int_\Omega u_\varepsilon v_\varepsilon,
    \end{align}
    for all $t \in (0,T_{\rm max}, \varepsilon)$ and $\varepsilon \in (0,1)$. Combining \eqref{M.1} and \eqref{M.2}, we obtain 
    \begin{align*}
        \frac{1}{p} \frac{d}{dt}\int_\Omega u_\varepsilon^p +\frac{p-1}{2}\int_\Omega u_\varepsilon^{p-1}v_\varepsilon |\nabla u_\varepsilon|^2 &\leq c_2 \int_\Omega u^{p_0-1}_\varepsilon v_\varepsilon |\nabla u_\varepsilon|^2 + (c_2+c_1) \int_\Omega u_\varepsilon \frac{|\nabla v_\varepsilon|^4}{v_\varepsilon^3} \notag \\
        &\quad+(c_2+\ell) \int_\Omega u_\varepsilon v_\varepsilon +c_1 \int_\Omega v_\varepsilon |\nabla v_\varepsilon|^2
    \end{align*}
    for all $t \in (0,T_{\rm max}, \varepsilon)$ and $\varepsilon \in (0,1)$. Integrating this in time and using \eqref{M-2}, \eqref{L3-1}, \eqref{local-5} and Lemma \ref{local-5'}, we can find $c_3>0$ such that
    \begin{align*}
        \frac{1}{p}\int_\Omega u_\varepsilon^p(\cdot,t) +\frac{p-1}{2}\int_0^t\int_\Omega u_\varepsilon^{p-1}v_\varepsilon |\nabla u_\varepsilon|^2 &\leq c_2\int_0^t \int_\Omega u^{p_0-1}_\varepsilon v_\varepsilon |\nabla u_\varepsilon|^2 + (c_2+c_1) \int_0^t\int_\Omega u_\varepsilon \frac{|\nabla v_\varepsilon|^4}{v_\varepsilon^3} \notag \\
        &\quad+(c_2+\ell) \int_0^t\int_\Omega u_\varepsilon v_\varepsilon +c_1 \int_0^t \int_\Omega v_\varepsilon |\nabla v_\varepsilon|^2 +\frac{1}{p}\int_\Omega (u_0+1)^p \\
        &\leq c_3 \qquad \text{for all }t\in (0,T_{\rm max,\varepsilon}) \quad \text{and }\varepsilon \in (0,1).
    \end{align*}
    The proof is now complete.
\end{proof}

As a consequence of the lemma above, we may now establish an $L^{p_0}$ bound for $u_\varepsilon$ for some $p_0>\frac{n}{2}$, by means of an iteration argument.
\begin{lemma} \label{L4}
      Let $\Omega \subset \mathbb{R}^n$ with $n \in \left \{ 3,4,5 \right \}$ be a bounded convex domain with smooth boundary and let $\alpha  \in \left [ 1, \frac{5}{2}- \frac{n}{4} \right )$. There exist $q_0>\frac{n}{2}$ and $C>0$ such that 
    \begin{align*}
        \int_\Omega u^{q_0}_\varepsilon(\cdot,t) \leq C \qquad \text{for all }t\in (0,T_{\rm max,\varepsilon}) \quad \text{and }\varepsilon \in (0,1).
    \end{align*}
\end{lemma}
\begin{proof}
   Putting $p_0=1$ and applying Lemma \ref{local}[\eqref{local-3}] and Lemma \ref{L3}[\eqref{L3-2}], we can find $c_1>0$ such that 
   \begin{align} \label{L4.1}
       \int_\Omega u^{p_0}_\varepsilon(\cdot,t) \leq c_1 \qquad \text{for all }t\in (0,T_{\rm max,\varepsilon}) \quad \text{and }\varepsilon \in (0,1)
   \end{align}
   and 
   \begin{align} \label{L4.2}
        \int_0^{T_{\rm max, \varepsilon}} \int_\Omega u_\varepsilon^{p_0-1} v_\varepsilon |\nabla u_\varepsilon|^2 \leq c_1 \qquad \text{for all } \varepsilon \in (0,1).
   \end{align}
   Setting $p_1= \min \left \{1+\frac{4}{n}p_0, \frac{2}{n}p_0+4-2\alpha \right \}$, using \eqref{L4.1} and \eqref{L4.2} together with Lemma \ref{M}, we can find $c_2>0$ such that 
     \begin{align*} 
       \int_\Omega u^{p_1}_\varepsilon(\cdot,t) \leq c_2 \qquad \text{for all }t\in (0,T_{\rm max,\varepsilon}) \quad \text{and }\varepsilon \in (0,1)
   \end{align*}
   and 
   \begin{align*} 
        \int_0^{T_{\rm max, \varepsilon}} \int_\Omega u_\varepsilon^{p_1-1} v_\varepsilon |\nabla u_\varepsilon|^2 \leq c_2 \qquad \text{for all } \varepsilon \in (0,1).
   \end{align*}
   By an inductive argument and Lemma \ref{M}, we deduce that for any $k \in \mathbb{N}$, there exists $c_{k+1}>0$ such that 
      \begin{align} \label{L4.3}
       \int_\Omega u^{p_k}_\varepsilon(\cdot,t) \leq c_{k+1} \qquad \text{for all }t\in (0,T_{\rm max,\varepsilon}) \quad \text{and }\varepsilon \in (0,1)
   \end{align}
   and 
   \begin{align*} 
        \int_0^{T_{\rm max, \varepsilon}} \int_\Omega u_\varepsilon^{p_k-1} v_\varepsilon |\nabla u_\varepsilon|^2 \leq c_{k+1} \qquad \text{for all } \varepsilon \in (0,1),
   \end{align*}
   where $p_k = \min \left \{1+\frac{4}{n}p_{k-1}, \frac{2}{n}p_{k-1}+4-2\alpha \right \}$. From the assumption $1\leq \alpha < \frac{5}{2}-\frac{n}{4}$, we find that 
   \begin{align*}
       \frac{2n(2-\alpha)}{n-2}>\frac{n}{2} \qquad \text{for }n \in \left \{3,4,5 \right \}.
   \end{align*}
   In view of Lemma \ref{D}, we have that 
   \begin{align*}
      1\leq  p_k<  \frac{2n(2-\alpha)}{n-2} \qquad \text{for all }k \geq 0
   \end{align*}
   and
   \begin{align*}
       \lim_{k \to \infty }p_k= \frac{2n(2-\alpha)}{n-2}.
   \end{align*}
    Therefore, we can find $k_0 \in \mathbb{N}$ such that 
   \begin{align*}
       \frac{2n(2-\alpha)}{(n-2)}>p_{k_0}>\frac{n}{2} \qquad \text{for }n \in \left \{3,4,5 \right \}.
   \end{align*}
   This, together with \eqref{L4.3} implies that there exists $C>0$ such that 
    \begin{align*}
        \int_\Omega u^{q_0}_\varepsilon(\cdot,t) \leq C \qquad \text{for all }t\in (0,T_{\rm max,\varepsilon}) \quad \text{and }\varepsilon \in (0,1),
    \end{align*}
    with $q_0=p_{k_0}>\frac{n}{2}$. The proof is now complete.
\end{proof}

Thanks to the $L^{\frac{n}{2}+}$ bound for $u_\varepsilon$ and the inequality established in Lemma \ref{PL2}, we can derive $L^p$ bounds for solutions of \eqref{2} for arbitrary $p>1$.

\begin{lemma} \label{Lp}
     Let $\Omega \subset \mathbb{R}^n$ with $n \in \left \{ 3,4,5 \right \}$ be a bounded convex domain with smooth boundary and let $\alpha \in \left [1, \frac{5}{2}-\frac{n}{4} \right )$. For any $p>1$, there exists $C>0$ such that 
    \begin{align}\label{Lp-1}
        \int_\Omega u^p_\varepsilon(\cdot,t) \leq C \qquad \text{for all }t\in (0,T_{\rm max, \varepsilon}) \quad \text{and }\varepsilon \in (0,1)
    \end{align}
    and 
    \begin{align} \label{Lp-2}
        \int_0^{T_{\rm max,\varepsilon}} \int_\Omega u_\varepsilon^{p-1} v_\varepsilon |\nabla u_\varepsilon|^2 \leq C \qquad \text{for all } \varepsilon \in (0,1)
    \end{align}
    as well as 
    \begin{align} \label{Lp-3}
         \int_0^{T_{\rm max,\varepsilon}}  \int_\Omega \frac{|\nabla v_\varepsilon|^{2p+4}} {v_\varepsilon^{2p+3}}\qquad \text{for all } \varepsilon \in (0,1).
    \end{align}
\end{lemma}
\begin{proof}
    From Lemma \ref{L4}, there exists $p_0>\frac{n}{2}$ such that 
     \begin{align}\label{Lp.1}
        \int_\Omega u^{p_0}_\varepsilon(\cdot,t) \leq c_1 \qquad \text{for all }t\in (0,T_{\rm max, \varepsilon}) \quad \text{and }\varepsilon \in (0,1)
    \end{align}
    and 
    \begin{align} \label{Lp.2}
        \int_0^{T_{\rm max,\varepsilon}} \int_\Omega u_\varepsilon^{p_0-1} v_\varepsilon |\nabla u_\varepsilon|^2 \leq c_1 \qquad \text{for all } \varepsilon \in (0,1),
    \end{align}
    for some $c_1>0$. Moreover, from Lemma \ref{L3}[\eqref{L3-2}], we can find $c_2>0 $ such that
    \begin{align}\label{Lp.3}
         \int_0^{T_{\rm max,\varepsilon}} \int_\Omega  v_\varepsilon |\nabla u_\varepsilon|^2 \leq c_2 \qquad \text{for all } \varepsilon \in (0,1).
    \end{align}
    Therefore, when $1<p \leq p_0$, we make use of \eqref{Lp.1}, \eqref{Lp.2} and \eqref{Lp.3} together with Young's inequality to deduce \eqref{Lp-1} and \eqref{Lp-2}. In case $p>p_0$, we test the first equation of \eqref{2} by $u_\varepsilon^{p-1}$ and apply Young's inequality to obtain
    \begin{align}\label{Lp.4}
       \frac{1}{p} \frac{d}{dt} \int_\Omega u_\varepsilon^p +\frac{p-1}{2}\int_\Omega u_\varepsilon^{p-1}v_\varepsilon |\nabla u_\varepsilon|^2 &\leq \frac{(p-1)\chi^2}{2}\int_\Omega u_\varepsilon ^{p+2\alpha-3}v_\varepsilon |\nabla v_\varepsilon|^2+ \ell \int_\Omega u_\varepsilon^p v_\varepsilon \notag \\
       &\leq  \int_\Omega u^{p+1}_\varepsilon v_\varepsilon |\nabla v_\varepsilon|^2+c_3 \int_\Omega v_\varepsilon |\nabla v_\varepsilon|^2+\ell \int_\Omega u_\varepsilon^p v_\varepsilon
    \end{align}
    for all $t\in (0,T_{\rm max, \varepsilon})$ and $\varepsilon \in (0,1)$, where $c_3>0$. Applying Lemma \ref{W}[\eqref{W-1}] with $q=2p+2$, there exist $c_3>0$ and $c_4>0$ such that 
    \begin{align}\label{Lp.5}
        \frac{d}{dt}\int_\Omega \frac{|\nabla v_\varepsilon|^{2p+2}}{v_\varepsilon^{2p+1}}+ c_3\int_\Omega u_\varepsilon \frac{|\nabla v_\varepsilon|^{2p+2}}{v_\varepsilon^{2p+1}}+c_3 \int_\Omega \frac{|\nabla v_\varepsilon|^{2p+4}}{v_\varepsilon^{2p+3}} \leq c_4 \int_\Omega u_\varepsilon^{p+2} v_\varepsilon
    \end{align}
     for all $t\in (0,T_{\rm max, \varepsilon})$ and $\varepsilon \in (0,1)$. In light of Young's inequality, we can find $c_5>0$ and $c_6>0$ such that
    \begin{align}\label{Lp.6}
        \int_\Omega u^{p+1}_\varepsilon v_\varepsilon |\nabla v_\varepsilon|^2+\ell \int_\Omega u_\varepsilon^p v_\varepsilon \leq \frac{c_3}{4} \int_\Omega \frac{|\nabla v_\varepsilon|^{2p+4}}{v_\varepsilon^{2p+3}}+c_5 \int_\Omega u_\varepsilon ^{p+2} v_\varepsilon +c_6\int_\Omega u_\varepsilon v_\varepsilon
    \end{align}
     for all $t\in (0,T_{\rm max, \varepsilon})$ and $\varepsilon \in (0,1)$. From \eqref{Lp.1}, we can apply Lemma \ref{PL2} with $\eta = \min \left \{ \frac{p-1}{4(c_3+c_5)}, \frac{c_3}{4(c_3+c_5)} \right \}$ to deduce that 
     \begin{align}\label{Lp.7}
         (c_3+c_5)\int_\Omega u_\varepsilon ^{p+2} v_\varepsilon &\leq \eta(c_3+c_5)\int_\Omega u_\varepsilon^{p-1}v_\varepsilon |\nabla u_\varepsilon|^2 + \eta(c_3+c_5)\int_\Omega u_\varepsilon \frac{|\nabla v_\varepsilon|^{2p+2}}{v_\varepsilon^{2p+1}}+c_7 \int_\Omega u_\varepsilon v_\varepsilon \notag \\
         &\leq \frac{p-1}{4}\int_\Omega u_\varepsilon^{p-1}v_\varepsilon |\nabla u_\varepsilon|^2 +\frac{c_3}{4}\int_\Omega u_\varepsilon \frac{|\nabla v_\varepsilon|^{2p+2}}{v_\varepsilon^{2p+1}}+c_7 \int_\Omega u_\varepsilon v_\varepsilon
     \end{align}
     for all $t\in (0,T_{\rm max, \varepsilon})$ and $\varepsilon \in (0,1)$, where $c_7>0$. Collecting from \eqref{Lp.4} to \eqref{Lp.7}, we arrive at 
     \begin{align*}
         &\frac{d}{dt} \left \{ \frac{1}{p}\int_\Omega u_\varepsilon^p +\int_\Omega \frac{|\nabla v_\varepsilon|^{2p+2}}{v_\varepsilon^{2p+1}} \right \} +\frac{p-1}{4} \int_\Omega u_\varepsilon^{p-1}v_\varepsilon |\nabla u_\varepsilon|^2 + \frac{c_3}{2} \int_\Omega \frac{|\nabla v_\varepsilon|^{2p+4}}{v_\varepsilon^{2p+3}}  \\
         &\leq  c_8 \int_\Omega u_\varepsilon v_\varepsilon+c_3 \int_\Omega v_\varepsilon |\nabla v_\varepsilon|^2  \qquad \text{for all }t\in (0,T_{\rm max, \varepsilon}) \quad \text{and }\varepsilon \in (0,1),
     \end{align*}
    where $c_8=c_6+c_7$. Integrating this over time and using \eqref{local-5} and Lemma \ref{local-5'}, we can find $c_9>0$ such that
    \begin{align*}
        &\frac{1}{p}\int_\Omega u_\varepsilon^p(\cdot,t) +\int_\Omega \frac{|\nabla v_\varepsilon(\cdot,t)|^{2p+2}}{v_\varepsilon^{2p+1}(\cdot,t)} +\frac{p-1}{4} \int_0^t \int_\Omega u_\varepsilon^{p-1}v_\varepsilon |\nabla u_\varepsilon|^2 + \frac{c_3}{2} \int_0^t \int_\Omega \frac{|\nabla v_\varepsilon|^{2p+4}}{v_\varepsilon^{2p+3}}  \notag \\
        &\leq c_8 \int_0^t\int_\Omega u_\varepsilon v_\varepsilon+c_3 \int_0^t \int_\Omega v_\varepsilon |\nabla v_\varepsilon|^2 +\frac{1}{p}\int_\Omega (u_0+1)^p + \int_\Omega \frac{|\nabla v_0|^{2p+2}}{v_0^{2p+1}} \leq c_9 
    \end{align*}
     for all $t\in (0,T_{\rm max, \varepsilon})$ and $\varepsilon \in (0,1)$, which implies \eqref{Lp-1}, \eqref{Lp-2} and \eqref{Lp-3} for $p> p_0$. Now, we only need to prove \eqref{Lp-3} with $1<p \leq p_0$. Indeed, by applying Young's inequality and employing \eqref{Lp-3} with $p=p_0+1$ and \eqref{L3-3}, we obtain 
     \begin{align*}
         \int_0^{T_{\rm max, \varepsilon}}\int_\Omega \frac{|\nabla v_\varepsilon|^{2p+4}}{v_\varepsilon^{2p+3}}  &\leq \int_0^{T_{\rm max, \varepsilon}}\int_\Omega \frac{|\nabla v_\varepsilon|^{2p_0+6}}{v_\varepsilon^{2p_0+5}} +\int_0^{T_{\rm max, \varepsilon}}\int_\Omega \frac{|\nabla v_\varepsilon|^{6}}{v_\varepsilon^{5}}   \notag \\
         &\leq c_{10} \qquad \text{for all }\varepsilon \in (0,1),
     \end{align*}
   with some $c_{10}>0$,  which completes the proof.
    
\end{proof}

\section{\texorpdfstring{$L^\infty$ boundedness of $u_\varepsilon$ and proof of the main results}{boundedness for u_epsilon}} \label{S5}

The purpose of this section is to derive an $L^\infty$ estimate for $u_\varepsilon$ and then apply it to establish the global solvability of weak solutions for the system \eqref{1}. Before going further, let us begin this section with an $L^\infty$ estimate for $\nabla v_\varepsilon$ in the following lemma.
\begin{lemma} \label{Unif-v}
  Let $\Omega \subset \mathbb{R}^n$ with $n \in \left \{ 3,4,5 \right \}$ be a bounded convex domain with smooth boundary and $\alpha \in \left [1, \frac{5}{2}-\frac{n}{4} \right )$. Then there exists $C>0$ such that 
 \begin{align*}
     \left \| \nabla v_\varepsilon(\cdot,t) \right \|_{L^\infty(\Omega)} \leq C \qquad \text{for all }t \in (0,T_{\rm max, \varepsilon}) \quad \text{and } \varepsilon \in (0,1).
 \end{align*}
\end{lemma}

\begin{proof}
    Applying \eqref{Lp-1} in Lemma \ref{Lp} with $p=2n$, we can find $c_1>0$ such that 
    \begin{align*}
        \int_\Omega u_\varepsilon^{2n}(\cdot,t) \leq c_1 \qquad \text{for all }t \in (0,T_{\rm max, \varepsilon}) \quad \text{and } \varepsilon \in (0,1).
    \end{align*}
    According to the smoothing properties of the Neumann heat semigroup theory (\cite{Winkler-2010}[Lemma 1.3]), we can find $c_2>0$ and $c_3>0$ such that 
    \begin{align*}
    \left \| \nabla v_\varepsilon(\cdot,t) \right \|_{L^\infty(\Omega)} &= \left \| \nabla e^{t(\Delta-1)} v_0 - \int_0^t \nabla e^{(t-s)(\Delta-1)} \left \{ u_\varepsilon(\cdot,s) v_\varepsilon(\cdot,s) - v_\varepsilon(\cdot,s) \right \} \right \|_{L^\infty(\Omega)} \, ds\notag \\
    &\leq c_2 \left \|v_0 \right \|_{W^{1, \infty}(\Omega)} +c_2 \int_0^t \left ( 1+(t-s)^{-\frac{3}{4}} \right ) e^{-(t-s)} \left \| u_\varepsilon(\cdot,s) v_\varepsilon(\cdot,s) - v_\varepsilon(\cdot,s) \right \|_{L^{2n}(\Omega)} \, ds \notag\\
    &\leq c_2 \left \|v_0 \right \|_{W^{1, \infty}(\Omega)} +c_2 \left \| v_0 \right \|_{L^\infty(\Omega)} \left ( c_1+ |\Omega|^{\frac{1}{2n}} \right )  \int_0^t \left ( 1+(t-s)^{-\frac{3}{4}} \right ) e^{-(t-s)}  \, ds \notag \\
    &\leq c_3 \qquad \text{for all }t \in (0,T_{\rm max, \varepsilon}) \quad \text{and } \varepsilon \in (0,1),
    \end{align*}
    which finishes the proof.
\end{proof}

Thanks to the $L^p$ bounds for $u_\varepsilon$ for arbitrary $p>1$ and the key inequality established in Lemma~\ref{Moser}, we can adapt the Moser iteration procedure---originally developed for the two-dimensional case in \cite{Winkler-2024}---to arbitrary dimensions in order to prove the uniform boundedness of $u_\varepsilon$ in the following lemma.

\begin{lemma} \label{Linf}
  Let $\Omega \subset \mathbb{R}^n$ with $n \in \left \{ 3,4,5 \right \}$ be a bounded convex domain with smooth boundary. Then one can find $C>0$ such that 
\begin{align} \label{Linf-1}
    \left \| u_\varepsilon(\cdot,t) \right \|_{L^\infty(\Omega)} \leq C \qquad \text{for all }t\in (0,T_{\rm max, \varepsilon}) \quad \text{and } \varepsilon \in (0,1).
\end{align}
\end{lemma}

\begin{proof}
    As a result of Lemma \ref{Unif-v}, we can find $c_1>0$ such that 
    \begin{align} \label{Linf.1}
        \left \| \nabla v_\varepsilon(\cdot,t) \right \|_{L^\infty(\Omega)} \leq c_1 \qquad \text{for all }t\in (0,T_{\rm max, \varepsilon}) \quad \text{and }\varepsilon \in (0,1).
    \end{align}
    For integers $k \geq 1$, we let $p_k:= 2^k n $ and set 
    \begin{align}\label{Linf.2}
        M_{k, \varepsilon} (T):= 1 + \sup_{t\in (0,T)} \int_\Omega u^{p_k}_\varepsilon (\cdot,t), \qquad  T \in (0,T_{\rm max,\varepsilon}), \, k \geq 1, \, \varepsilon \in (0,1),
    \end{align}
    then we see that $M_{k, \varepsilon} (T)$ is finite. Moreover, that \eqref{Lp-1} in Lemma \ref{Lp} ensures the existence of $c_2>0$ such that 
    \begin{align}\label{Linf.3}
        M_{0, \varepsilon}(T) \leq c_2 \qquad  \text{for all } T \in (0,T_{\rm max,\varepsilon}) \quad \text{ and }\varepsilon \in (0,1).
    \end{align}
    Testing the first equation of \eqref{2} by $p_k u_\varepsilon^{p_k-1}$, applying Young's inequality, using \eqref{Linf.1} and noting that $1< p_k+2\alpha-3 < p_k+1$, we obtain 
    \begin{align}\label{Linf.4}
        \frac{d}{dt} \int_\Omega u_\varepsilon^{p_k} &\leq - \frac{p_k(p_k-1)}{2} \int_\Omega u_\varepsilon^{p_k-1} v_\varepsilon |\nabla u_\varepsilon|^2 +\frac{p_k(p_k-1)\chi^2}{2} \int_\Omega u_\varepsilon^{p_k+2\alpha-3} v_\varepsilon |\nabla v_\varepsilon|^2 + p_k\ell \int_\Omega u_\varepsilon^{p_k}v_\varepsilon \notag \\
        &\leq - \frac{p_k^2}{4}\int_\Omega u_\varepsilon^{p_k-1} v_\varepsilon |\nabla u_\varepsilon|^2 + \frac{p_k^2c_1^2\chi^2}{2} \int_\Omega u_\varepsilon^{p_k+2\alpha-3} v_\varepsilon + p_k\ell \int_\Omega u_\varepsilon^{p_k}v_\varepsilon \notag \\
        &\leq  -\frac{p_k^2}{4}\int_\Omega u_\varepsilon^{p_k-1} v_\varepsilon |\nabla u_\varepsilon|^2 + c_3 p_k^2 \int_\Omega u_\varepsilon^{p_k+1} v_\varepsilon + c_3 p_k^2  \int_\Omega u_\varepsilon v_\varepsilon,
    \end{align}
    where $c_3= \frac{c_1^2\chi^2}{2}+\ell$. Since $p_*:=2n>n$, we can make use of Lemma \ref{Moser} with $p=p_k$ to deduce that 
    \begin{align}\label{Linf.5}
        \int_\Omega u^{p_k+1}_\varepsilon v_\varepsilon &\leq \frac{1}{4c_3} \int_\Omega u_\varepsilon^{p_k-1} v_\varepsilon |\nabla u_\varepsilon|^2 +\frac{1}{4c_3} \cdot \left \{ \int_\Omega u_\varepsilon^{\frac{p_k}{2}} \right \}^{\frac{2(p_k+1)}{p_k}} \cdot \int_\Omega \frac{|\nabla v_\varepsilon|^{2n+2}}{v_\varepsilon^{2n+1} }  \notag \\
        &\quad+K\cdot (4c_3)^\kappa \cdot p_k^{2\kappa} \cdot \left \{ \int_\Omega u_\varepsilon^{\frac{p_k}{2}} \right \}^2 \cdot \int_\Omega u_\varepsilon v_\varepsilon \qquad \text{for all }t\in (0,T_{\rm max, \varepsilon}),
    \end{align}
    where $\kappa = \kappa(p_*)>0$ and $K=K(p_*)>0$. Recalling the definition of $(p_j)_{j \geq 1}$  and \eqref{Linf.2}, and using \eqref{Linf.4} and \eqref{Linf.5}, we deduce that 
    \begin{align}\label{Linf.6}
         \frac{d}{dt} \int_\Omega u_\varepsilon^{p_k} &\leq \frac{p_k^2}{4} \cdot \left \{ \int_\Omega u_\varepsilon^{\frac{p_k}{2}} \right \}^{\frac{2(p_k+1)}{p_k}} \cdot \int_\Omega \frac{|\nabla v_\varepsilon|^{2n+2}}{v_\varepsilon^{2n+1} } +c_3p_k^2 \int_\Omega u_\varepsilon v_\varepsilon  \notag \\
        &\quad+K 4^\kappa  c_3^{\kappa+1}  p_k^{2\kappa+2} \cdot \left \{ \int_\Omega u_\varepsilon^{\frac{p_k}{2}} \right \}^2 \cdot \int_\Omega u_\varepsilon v_\varepsilon \notag \\
        &\leq \frac{p_k^2}{4} M_{k-1, \varepsilon}^{\frac{2(p_k+1)}{p_k}}(T)  \int_\Omega \frac{|\nabla v_\varepsilon|^{2n+2}}{v_\varepsilon^{2n+1} } +c_3 p_k^{2} \int_\Omega u_\varepsilon v_\varepsilon  \notag \\
          &\quad+K 4^\kappa  c_3^{\kappa+1}  p_k^{2\kappa+2} M_{k-1, \varepsilon}^{2}(T)  \int_\Omega u_\varepsilon v_\varepsilon \notag \\
          &\leq c_4p_k^{2\kappa+2} M_{k-1, \varepsilon}^{\frac{2(p_k+1)}{p_k}}(T) \left \{ \int_\Omega \frac{|\nabla v_\varepsilon|^{2n+2}}{v_\varepsilon^{2n+1} } +\int_\Omega u_\varepsilon v_\varepsilon \right \} 
    \end{align}
  for all $t\in (0,T)$, any $T\in (0,T_{\rm max,\varepsilon})$ and each $\varepsilon \in (0,1)$,  where $c_4= \max \left \{ \frac{1}{4}, K 4^\kappa c_3^{\kappa+1} , c_3 \right \}$. Applying \eqref{Lp-3} with $p:= n-1$ and \eqref{local-5} yields $c_5>0$ such that 
  \begin{align*}
      \int_0^{T_{\rm max,\varepsilon }} \int_\Omega \frac{|\nabla v_\varepsilon|^{2n+2}}{v_\varepsilon^{2n+1} }+ \int_0^{T_{\rm max,\varepsilon }}\int_\Omega u_\varepsilon v_\varepsilon \leq c_5 \qquad \text{for all }\varepsilon \in (0,1).
  \end{align*}
  This, together with an integration of \eqref{Linf.6} in time implies that for all $t\in (0,T)$, $T\in (0,T_{\rm max, \varepsilon})$ and $\varepsilon \in (0,1)$, 
  \begin{align*}
      \int_\Omega u^{p_k}(\cdot,t) &\leq  \int_\Omega (u_0+\varepsilon)^{p_k} +c_4c_5 p_k^{2\kappa+2}    M_{k-1, \varepsilon}^{\frac{2(p_k+1)}{p_k}}(T) \notag \\
      &\leq |\Omega| \left \| u_0+1 \right \|_{L^\infty(\Omega)}^{p_k}+c_4c_5 p_k^{2\kappa+2}    M_{k-1, \varepsilon}^{\frac{2(p_k+1)}{p_k}}(T).
  \end{align*}
  In view of \eqref{Linf.2}, this entails that 
  \begin{align*}
      M_{k, \varepsilon}(T) &\leq  1+|\Omega| \left \| u_0+1 \right \|_{L^\infty(\Omega)}^{p_k}+c_4c_5 p_k^{2\kappa+2}    M_{k-1, \varepsilon}^{\frac{2(p_k+1)}{p_k}}(T) \\
      &\leq a^{2^k} +b^k       M_{k-1, \varepsilon}^{2+\frac{d}{2^k}}(T) \qquad \text{for all }T\in (0,T_{\rm max, \varepsilon}) \quad \text{and }\varepsilon \in (0,1)
  \end{align*}
  with 
  \begin{align*}
      a:= \left (1+ |\Omega| \left \| u_0+1 \right \|_{L^\infty(\Omega)} \right )^n, \qquad b:= \left ( \max \left \{2c_4c_5n,1 \right \} \right )^{2\kappa+2} \qquad \text{and } d:= \frac{2}{n},
  \end{align*}
  because by the definition of $(p_k)_{k \geq 1}$,
  \begin{align*}
       1+|\Omega| \left \| u_0+1 \right \|_{L^\infty(\Omega)}^{p_k} \leq \left (   1+|\Omega| \left \| u_0+1 \right \|_{L^\infty(\Omega)} \right )^{p_k} = \left \{ \left (   1+|\Omega| \left \| u_0+1 \right \|_{L^\infty(\Omega)} \right )^{n} \right \}^{2^k}
  \end{align*}
  and 
  \begin{align*}
      c_4c_5p_k^{2\kappa+2} = c_4c_5 ( 2^k n)^{2\kappa+2}  \leq  \left \{ \left ( \max \left \{2 c_4c_5n,1 \right \} \right )^k \right \}^{2\kappa+2} =   \left \{ \left ( \max \left \{2 c_4c_5n,1 \right \} \right )^{2\kappa+2} \right \}^{k}
  \end{align*}
  as well as 
  \begin{align*}
      \frac{2(p_k+1)}{p_k} = 2+\frac{2}{p_k}=2+ \frac{2}{n} 2^{-k}.
  \end{align*}
  Now, applying Lemma \ref{MW} and using \eqref{Linf.3}, we arrive at 
  \begin{align*}
      \left \| u_\varepsilon (\cdot,t) \right \|_{L^\infty(\Omega)}^n &= \liminf_{k \to \infty} \left \{ \int_\Omega u_\varepsilon^{p_k}(\cdot,t) \right \}^{\frac{1}{2^k}} \notag \\
      &\leq \liminf_{k \to \infty }M_{k,\varepsilon}^{\frac{1}{2^k}}(T) \notag \\
      &\leq \left ( 2\sqrt{2}b^3 a^{1+\frac{d}{2}} c_2 \right )^{e^{\frac{d}{2}}} 
  \end{align*}
  for all $t\in (0,T)$, any $T \in (0,T_{\rm max, \varepsilon})$ and each $\varepsilon \in (0,1)$. Finally, taking $T \nearrow T_{\rm max, \varepsilon}$, we deduce \eqref{Linf-1}. The proof is now complete.  
\end{proof}

As a result of the above lemma, we demonstrate in the following lemma that the classical solutions of \eqref{2} exist globally in time.
    \begin{lemma} \label{Global}
        Under the assumption of Theorem \ref{thm1}, then for each $\varepsilon \in (0,1)$ we have $T_{\rm max, \varepsilon} = \infty$.
    \end{lemma}
    \begin{proof}
        This is an immediate consequence of Lemma \ref{Linf} and the extensibility property of solution \eqref{local-2}.
    \end{proof}

Moreover, the $L^\infty$ bounds for $u_\varepsilon$ established in Lemma \ref{Linf} enables us to derive a lower bound for $v_\varepsilon$ by means of simple comparison argument.

 \begin{lemma} \label{lowv}
       Under the assumption of Theorem \ref{thm1}, then for all $T\in (0,\infty)$ there exists $C(T)>0$ such that 
       \begin{align} \label{lowv-1}
           v_\varepsilon(x,t) \geq C(T) \qquad \text{for all }x\in \Omega, t\in (0,T) \, \text{and }\varepsilon \in (0,1).
       \end{align}
 \end{lemma}
\begin{proof}
    From Lemma \ref{Linf} and Lemma \ref{Global}, we can find $c_1>0$ such that 
    \begin{align*}
        \left \| u_\varepsilon (\cdot,t) \right \|_{L^\infty(\Omega)} \leq c_1 \qquad \text{for all }t>0 \quad \text{and } \varepsilon\in (0,1),
    \end{align*}
    which entails that 
    \begin{align*}
        v_{\varepsilon t} \geq \Delta v_\varepsilon -c_1 v_\varepsilon \quad \text{in } \Omega \times (0,\infty) \qquad \text{for all }\varepsilon \in (0,1).
    \end{align*}
    An application of comparison principle yields that 
    \begin{align} \label{lowv.1}
        v_\varepsilon(x,t) \geq e^{-c_1t } \inf_{x\in \Omega}v_0  \geq e^{-c_1T}\inf_{x\in \Omega}v_0 \qquad \text{for all }x\in \Omega, t\in (0,T) \, \text{and }\varepsilon \in (0,1),
    \end{align}
    which completes the proof.
\end{proof}

Thanks to the above estimates together with standard parabolic theory, we directly obtain the H\"older regularity results for the global solution of \eqref{2}.

\begin{lemma} \label{Schauder}
    Under the assumption of Theorem \ref{thm1}, for all $T>0$ and $\varepsilon \in (0,1)$, there exist $\theta_1= \theta_1(T) \in (0,1)$ and $C_1(T)>0$ such that 
    \begin{align} \label{Schauder-1}
        \left \| u_\varepsilon \right \|_{C^{\theta_1, \frac{\theta_1}{2}}(\Bar{\Omega}\times [0,T])} \leq C_1(T) \qquad \text{for all }\varepsilon \in (0,1)
    \end{align}
    and 
    \begin{align} \label{Schauder-2}
        \left \| v_\varepsilon \right \|_{C^{\theta_1, \frac{\theta_1}{2}}(\Bar{\Omega}\times [0,T])} \leq C_1(T) \qquad \text{for all }\varepsilon \in (0,1).
    \end{align}
    Moreover, for all $\tau >0$ and $T>\tau $, one can find  $\theta_2= \theta_2( \tau, T) \in (0,1)$ and $C_2(T)>0$ such that 
    \begin{align} \label{Schauder-3}
        \left \| v_\varepsilon \right \|_{C^{2+\theta_2, 1+\frac{\theta_2}{2}}(\Bar{\Omega}\times [\tau ,T])} \leq C_2(\tau, T) \qquad \text{for all }\varepsilon \in (0,1).
    \end{align}
\end{lemma}
\begin{proof}
    The proof of \eqref{Schauder-1} and \eqref{Schauder-2} can be proven by employing Lemmas \ref{lowv}, \ref{Unif-v}, \ref{Linf} and \eqref{local-4} and an application of the parabolic H\"older regularity theory as established in Theorem 1.3 in \cite{PV}. To prove \eqref{Schauder-3}, we just need to apply the classical parabolic Schauder theory (\cite{LSU}) together with the H\"older continuity of $u_\varepsilon$ established in \eqref{Schauder-1}.
\end{proof}

Combining all the estimates above, we may now apply a standard extraction argument to construct a limit function, which in fact yields a global weak solution as stated in Theorem~\ref{thm1}.

\begin{lemma} \label{conv}
      Under the assumption of Theorem \ref{thm1}, then there exist $(\varepsilon_j)_{j \in \mathbb{N}}\subset (0,1)$ as well as a pair $(u,v)$ of functions fulfilling \eqref{thm1-1} and \eqref{thm1-2}, such that $\varepsilon_j \searrow 0$ as $j \to \infty$, that $u \geq 0$ and $v>0$ in $\Bar{\Omega}\times [0, \infty)$, that
      \begin{align}
          &u_\varepsilon \to u \qquad \text{in } C_{loc}^0(\Bar{\Omega}\times [0, \infty) ), \label{conv-1}\\
          &v_\varepsilon \to v\qquad \text{in } C_{loc}^0(\Bar{\Omega}\times [0, \infty) ) \cap C^{2,1}(\Bar{\Omega}\times (0, \infty) ), \label{conv-2}\\
          &\nabla v_\varepsilon \stackrel{*}{\rightharpoonup} \nabla v \qquad \text{in }L^\infty \left ( \Omega \times (0,\infty) \right ) \label{conv-3}
      \end{align}
      as $\varepsilon = \varepsilon_j \searrow 0$, and that $(u,v)$ forms a global weak solution of \eqref{1} in the sense of Definition \ref{Def}.
 \end{lemma}
\begin{proof}
    In view of Lemma \ref{Lp}, Lemma \ref{lowv} and Young's inequality, for any $T>0$, there exists $C(T)>0$ such that 
    \begin{align*}
        \int_0^T \int_\Omega |\nabla u_\varepsilon^2| &\leq \int_0^T \int_\Omega v_\varepsilon |\nabla u_\varepsilon|^2 + \int_0^T \int_\Omega \frac{u^2_\varepsilon}{v_\varepsilon}  \notag \\
        &\leq C(T) \qquad \text{for all }\varepsilon \in (0,1),
    \end{align*}
  which yields that
    \begin{align*}
        \left ( u_\varepsilon \right )_{\varepsilon \in (0,1)} \text{is bounded in } L^1 \left ( (0,T); W^{1,1}(\Omega) \right ).
    \end{align*}
    Subsequently, we make use of Lemma \ref{Unif-v}, Lemma \ref{Schauder} and the Arzel\`a--Ascoli theorem to have a sequence $(\varepsilon_j)_{j \in \mathbb{N}}$ such that $\varepsilon_j \searrow 0$ as $j \to \infty$, and that with some function $u$ and $v$ satisfying \eqref{thm1-1} and \eqref{thm1-2}, besides \eqref{conv-1}, \eqref{conv-2}, \eqref{conv-3}, we have $\nabla u_\varepsilon^2 \to \nabla u^2$ in $L^1_{loc}\left ( \Bar{\Omega}\times [0, \infty) \right )$ as $\varepsilon= \varepsilon_j \searrow 0$, hence particularly implying that \eqref{Def.1} and \eqref{Def.2} hold. In addition, \eqref{Def.3} and \eqref{Def.4} can be accomplished on the basis of these convergence properties by taking $\varepsilon=\varepsilon_j \searrow 0$ in the respective weak formulations associated with \eqref{2}.
\end{proof}

We are now in position to prove our first main result.
\begin{proof}[Proof of Theorem \ref{thm1}]
Theorem \ref{thm1} is a immediate consequence of Lemma \ref{conv}.
\end{proof}

\section{Large time behavior of solutions} \label{S6}

We now establish the large-time behavior of solutions to \eqref{1}, which can be shown by adapting the proof of \cite[Theorem 1.2]{Wu2026}. For completeness, however, we include the full argument here. First of all, thanks to the uniform boundedness of $u_\varepsilon$ as established in Lemma \ref{Linf}, we can derive the following pointwise Harnack inequality where its proof can be found in Lemma 4.1 in \cite{Wu2026}.
\begin{lemma}\label{Hanack}
    Under the assumptions of Theorem \ref{1}, then there exists $\lambda>0$ such that 
    \begin{align*}
        v_\varepsilon(x,t) \geq \lambda \left \|v_\varepsilon(\cdot,t) \right \|_{L^\infty(\Omega)} \qquad \text{for all }x \in \Omega, t>0 \quad \text{and }\varepsilon \in (0,1).
    \end{align*}
\end{lemma}

As a consequence, we can deduce that $v_\varepsilon \in L^1 \left ((0, \infty);L^\infty(\Omega) \right )$ in the following lemma.
\begin{lemma}\label{S6L1}
     Under the assumptions of Theorem \ref{1}, then we have
     \begin{align*}
         \int_{t_0}^\infty \left \| v_\varepsilon(\cdot,t) \right \|_{L^\infty(\Omega)} dt \leq  \frac{\int_\Omega v_\varepsilon(\cdot,t_0)}{\lambda \int_\Omega u_0} \qquad \text{for all }t_0 \geq 0 \quad\text{and } \varepsilon \in (0,1),
     \end{align*}
     where $\lambda>0$ is given in Lemma \ref{Hanack}.
\end{lemma}
\begin{proof}
    From \eqref{local-5}, \eqref{local-3} and Lemma \ref{Hanack}, we infer that 
    \begin{align}
        \int_\Omega v_\varepsilon (\cdot,t_0) &\geq \int_{t_0}^\infty \int_\Omega u_\varepsilon(\cdot,t) v_\varepsilon(\cdot,t) \, dt \geq \lambda \int_{t_0}^\infty \left \|v_\varepsilon(\cdot,t) \right \|_{L^\infty} \int_\Omega u_\varepsilon(\cdot,t) \, dt \notag \\
        &\geq \lambda  \left \{\int_\Omega u_0 \right \} \int_{t_0}^\infty \left \|v_\varepsilon(\cdot,t) \right \|_{L^\infty} \, dt,
    \end{align}
    which finishes the proof.
\end{proof}

Using the preceding estimates, we can now establish that, along the sequence from Lemma~\ref{conv}, 
\(v_\varepsilon \to v\) in \( L^1 \left ((0, \infty);L^\infty(\Omega) \right )\).

\begin{lemma}\label{S6L2}
    Suppose that the assumptions of Theorem \ref{thm1} are satisfied and $(u_\varepsilon, v_\varepsilon)$ and $(\varepsilon_j)_{j \in \mathbb{N}}$ are given in Lemma \ref{conv}. Let 
    \begin{align*}
        L_\varepsilon := \int_0^\infty \left \| v_\varepsilon(\cdot,t) \right \|_{L^\infty(\Omega)}\, dt, \qquad \varepsilon \in (\varepsilon_j)_{j \in \mathbb{N}},
    \end{align*}
    and
    \begin{align*}
        \tau := \phi_\varepsilon(t):= \frac{1}{L_\varepsilon}\int_0^t \left \| v_\varepsilon(\cdot,s) \right \|_{L^\infty(\Omega)}\, ds, \qquad t\geq 0.
    \end{align*}
    Then, 
    \begin{align}\label{S6L2-1}
        w_\varepsilon(x,\tau ):= u_\varepsilon(x, \phi^{-1}_\varepsilon(\tau )), \qquad x\in \bar{\Omega}, \tau \in [0,1)
    \end{align}
    solves 
    \begin{equation}\label{S6L2-2}
        \begin{cases}
            w_{\varepsilon \tau } = \nabla \cdot \left ( a_\varepsilon(x, \tau ) w_\varepsilon \nabla w_\varepsilon \right )- \nabla \cdot \left (b_\varepsilon(x,\tau ) w_\varepsilon^\alpha \right ) +\ell a_\varepsilon(x,\tau ) w_\varepsilon, &\quad x\in \Omega, \tau \in (0,1), \\
            \nabla w_\varepsilon \cdot \nu =0 , &\quad x\in \Omega, \tau \in (0,1), \\
            w_\varepsilon(x,0)= u_0(x)+\varepsilon , &\quad x\in \Omega
        \end{cases}
    \end{equation}
    with 
    \begin{align*}
        a_\varepsilon(x, \tau ):= L_\varepsilon \cdot \frac{v_\varepsilon(x,t)}{ \left \| v_\varepsilon(\cdot,t) \right \|_{L^\infty(\Omega)}} \quad \text{and }\quad b_\varepsilon(x, \tau ):= L_\varepsilon \cdot \frac{v_\varepsilon(x,t)\nabla v_\varepsilon(x,t)}{\left \| v_\varepsilon(\cdot,t) \right \|_{L^\infty(\Omega)}}.
    \end{align*}
    Moreover, there exists $C>0$ such that
    \begin{align}\label{S6L2-3}
        \frac{1}{C} \leq a_\varepsilon(x,\tau ) \leq C \quad \text{and }\quad |b_\varepsilon(x,\tau )| \leq C \quad \text{for all } (x,\tau )\in \Omega \times (0,1) \quad \text{and }\quad \varepsilon \in (\varepsilon_j)_{j \in \mathbb{N}},
    \end{align}
    and
    \begin{align} \label{S6L2-4}
         L:= \int_0^\infty \left \| v(\cdot,t) \right \|_{L^\infty(\Omega)} \leq C
    \end{align}
    as well as
    \begin{align} \label{S6L2-5}
        L_\varepsilon \to L\quad \text{as }   \quad \varepsilon=\varepsilon_j \to 0.
    \end{align}
\end{lemma}
\begin{proof}
By Lemma \ref{S6L1}, we find that
\begin{align} \label{S6L2.1}
    L_\varepsilon = \int_0^\infty \left \| v_\varepsilon(\cdot,t) \right \|_{L^\infty(\Omega)}\, dt \leq c_1:= \frac{\int_\Omega v_0}{\lambda \int_\Omega u_0} \qquad \text{for all } \varepsilon \in  (\varepsilon_j)_{j \in \mathbb{N}} ,
\end{align}
which together with \eqref{conv-2} and Fatou's lemma implies \eqref{S6L2-4} since 
\begin{align*}
    L = \int_0^\infty \left \| v(\cdot,t) \right \|_{L^\infty(\Omega)}\, dt \leq \liminf_{\varepsilon \to 0} \int_0^\infty \left \| v_\varepsilon(\cdot,t) \right \|_{L^\infty(\Omega)}\, dt \leq  c_1 .
\end{align*}
Using \eqref{S6L2.1} and Lemma \ref{Unif-v}, we obtain 
\begin{align}\label{S6L2.2}
    a_\varepsilon(x, \tau ) \leq L_\varepsilon \leq c_1 \qquad \text{for all } (x, \tau) \in \Omega \times(0,1)\quad \text{and }\varepsilon \in  (\varepsilon_j)_{j \in \mathbb{N}} ,
\end{align}
and 
\begin{align}\label{S6L2.3}
    |b_\varepsilon(x, \tau)| \leq L_\varepsilon \left \| \nabla v_\varepsilon (\cdot,t) \right \|_{L^\infty(\Omega)} \leq c_2 \qquad \text{for all } \varepsilon \in  (\varepsilon_j)_{j \in \mathbb{N}} ,
\end{align}
where $c_2>0$. Recalling the estimate \eqref{lowv.1}, we can find $c_3>0$ and $c_4>0$ such that 
\begin{align*}
    v_\varepsilon(x,t) \geq c_3e^{-c_4t} \qquad \text{for all }x\in \Omega, t>0\, \text{ and } \varepsilon \in  (\varepsilon_j)_{j \in \mathbb{N}},
\end{align*}
which combines with Lemma \ref{Hanack} entails that 
\begin{align}\label{S6L2.4}
    a_\varepsilon(x,\tau ) &\geq \lambda \int_0^\infty \left \| v_\varepsilon(\cdot,t) \right \|_{L^\infty(\Omega)}\, dt \notag \\
    &\quad\geq \lambda c_3 \int_0^\infty e^{-c_4t}\, dt= \frac{\lambda c_3}{c_4} \qquad \text{for all } (x, \tau) \in \Omega \times(0,1)\quad \text{and }\varepsilon \in  (\varepsilon_j)_{j \in \mathbb{N}}.
\end{align}
Collecting \eqref{S6L2.2}, \eqref{S6L2.3} and \eqref{S6L2.4}, we obtain \eqref{S6L2-3}. From \eqref{conv-2}, we have that $\left \|v(\cdot, t) \right \|_{L^1(\Omega)}$ is uniformly continuous. Moreover, because
\begin{align*}
    \int_0^\infty \left \| v(\cdot,t) \right \|_{L^1(\Omega)} &\leq |\Omega|   \int_0^\infty \left \| v(\cdot,t) \right \|_{L^\infty(\Omega)} =L|\Omega|, 
\end{align*}
we apply Lemma \ref{Lunif} to deduce that for any $\eta>0$, there exists $t_0>0$ such that 
\begin{align*}
    \left \| v(\cdot,t_0) \right \|_{L^1(\Omega)} \leq \frac{\eta \lambda \int_\Omega u_0}{6},
\end{align*}
which together with \eqref{conv-2} implies that there exists $\varepsilon_*\in (0,1)$ satisfying 
\begin{align*}
    \left \| v_\varepsilon(\cdot,t_0) \right \|_{L^1(\Omega)} \leq \frac{\eta \lambda \int_\Omega u_0}{3} \qquad\text{for all } \varepsilon \in (\varepsilon_j)_{j \in \mathbb{N}} \quad \text{with } \varepsilon< \varepsilon_*.
\end{align*}
This, in conjunction with Lemma \ref{S6L1} entails that 
\begin{align} \label{S6L2.5}
    \int_{t_0} ^\infty  \left \| v_\varepsilon(\cdot,t) \right \|_{L^\infty(\Omega)} \, dt \leq \frac{\int_\Omega v_\varepsilon(\cdot,t_0)}{\lambda \int_\Omega u_0} \leq \frac{\eta}{3} \qquad\text{for all } \varepsilon \in (\varepsilon_j)_{j \in \mathbb{N}} \quad \text{with } \varepsilon< \varepsilon_*.
\end{align}
Applying \eqref{conv-2}, \eqref{S6L2.5} and Fatou's lemma, we deduce that 
\begin{align} \label{S6L2.6}
     \int_{t_0} ^\infty  \left \| v(\cdot,t) \right \|_{L^\infty(\Omega)} \, dt \leq  \frac{\eta}{3}.
\end{align}
and moreover, we can choose $\varepsilon_{**}\in (0, \varepsilon_*)$ fulfilling
\begin{align*}
    \left \| v_\varepsilon(\cdot,t)-v(\cdot,t) \right \|_{L^\infty(\Omega)} \leq \frac{\eta}{3t_0} \qquad\text{for all } \varepsilon \in (\varepsilon_j)_{j \in \mathbb{N}} \quad \text{with } \varepsilon< \varepsilon_{**}.
\end{align*}
This, together with \eqref{S6L2.5} and \eqref{S6L2.6} implies that
\begin{align*}
    |L_\varepsilon -L|&= \left |  \int_0 ^\infty \left \| v_\varepsilon (\cdot,t)\right \|_{L^\infty(\Omega)}\, dt- \int_0 ^\infty \left \| v (\cdot,t)\right \|_{L^\infty(\Omega)}\, dt \right | \\
    &\leq \int_0^{t_0} \left \| v_\varepsilon (\cdot,t)- v(\cdot,t)\right \|_{L^\infty(\Omega)}\, dt + \int_{t_0} ^\infty \left \| v_\varepsilon (\cdot,t)\right \|_{L^\infty(\Omega)}\, dt+ \int_{t_0} ^\infty \left \| v (\cdot,t)\right \|_{L^\infty(\Omega)}\, dt \\
    &\leq \eta \qquad\text{for all } \varepsilon \in (\varepsilon_j)_{j \in \mathbb{N}} \quad \text{with } \varepsilon< \varepsilon_{**},
\end{align*}
which proves \eqref{S6L2-5}. The proof is now complete.
\end{proof}
We are now in position to prove our second main result.
\begin{proof}[Proof of Theorem \ref{thm2}]
    Thanks to \eqref{conv-2} and \eqref{local-4}, we find that $ \left \| v(\cdot,t) \right \|_{L^\infty(\Omega)} $ is monotone decreasing with respect to $t>0$. This, in conjunction with \eqref{S6L2-4} implies that
    \begin{align} \label{thm2.2}
        \left \| v(\cdot,t) \right \|_{L^\infty(\Omega)}  \to 0 \qquad \text{as }t \to \infty.
    \end{align}
Using \eqref{conv-2} and \eqref{S6L2-5}, we deduce that 
    \begin{align*}
        \phi_\varepsilon \to \phi \qquad \text{uniformly on }[0,T], \quad \text{for all }T>0 \quad \text{as } \varepsilon = \varepsilon_j \to 0,
    \end{align*}
    where $(\varepsilon_j)_{j \in \mathbb{N}}$ is given in Lemma \ref{conv}. Hence, from \eqref{conv-1} and \eqref{conv-2}, we have
    \begin{align} \label{thm2.1}
        w_\varepsilon(x, \tau) \to u(x, \phi^{-1}(\tau)), \quad a_\varepsilon(x,\tau) \to a(x,\tau) \quad \text{and } b_\varepsilon(x,\tau) \to b(x, \tau)  
    \end{align}
     for all $(x, \tau ) \in \Omega \times (0,1)$ as $\varepsilon = \varepsilon_j \searrow 0$. Now, we make use of \eqref{S6L2-3} and employ Theorem 1.3 in \cite{PV} to deduce that there exists $\theta \in (0,1)$ and $c_1>0$ satisfying
     \begin{align*}
          \left \| w_\varepsilon \right \|_{C^{\theta,\frac{\theta}{2}}\left ( \bar{\Omega}\times [0,1] \right )} \leq c_1 \qquad \text{for all } \varepsilon \in (0,1).
     \end{align*}
     This enables us to apply the Arzel\`a--Ascoli theorem  to obtain that 
     \begin{align*}
         w_\varepsilon(x, \tau) \to w(x, \tau ) \quad \text{in } C^0 \left ( \Bar{\Omega}\times [0,1] \right ) \, \text{as } \varepsilon = \varepsilon_j \searrow 0,
     \end{align*}
     which together with \eqref{conv-2} entails that 
     \begin{align*}
         w(x, \tau )= u(x, \phi^{-1}(\tau ))\qquad \text{for all }(x, \tau) \in \Omega \times (0,1).
     \end{align*}
     Because $w(\cdot, 1)$ is continuous in $\bar{\Omega}$, we infer that 
     \begin{align*}
         u(\cdot,t) \to u_\infty:= w(\cdot,1) \quad \text{in } \Omega \quad \text{as } t \to \infty,
     \end{align*}
     which combines with \eqref{thm2.2} implies \eqref{thm2-1}. Since $u^2 \in L^1_{loc} \left ([0, \infty); W^{1,1}(\Omega) \right ) $, we claim that $w^2 \in L^1_{loc}\left ([0, 1); W^{1,1}(\Omega) \right )$.  Moreover, \eqref{thm2-2} is a consequence of \eqref{Def.3} and \eqref{thm2-3} is a result of \eqref{S6L2-3} and \eqref{thm2.1}. The proof is now complete. 
     
\end{proof}

\section*{Declarations}
\paragraph{Conflict of Interest} The authors declare that they have no conflict of interest.
\paragraph{Declaration of Generative AI and AI-Assisted Technologies in the Manuscript Preparation Process}

During the preparation of this work, the author(s) used \textbf{DeepSeek} (\texttt{deepseek.com}) for language polishing, grammar correction, and overall readability improvement. The author(s) reviewed and edited the output as needed and take full responsibility for the content of the published article.
\paragraph{Acknowledgments} Minh Le was supported by the Hangzhou Postdoctoral Research Grant.

     \paragraph{Data Availability}
 Data sharing not applicable to this article as no datasets were generated or analyzed during
the current study.

\end{document}